\documentclass{amsart}
\usepackage{booktabs}
\usepackage{fullpage}
\usepackage{tristan}
\usepackage{algorithm}
\usepackage{algpseudocode}
\usepackage{todonotes}

\usepackage{pgfplots}
\usepackage{pgfplotstable}
\usepackage{subcaption}
\usepackage{tikz}
\usepackage[
  giveninits=true,
  eprint=false,
  url=false,
  style=alphabetic,
  maxbibnames=99
]{biblatex}
\definecolor{seaborngreen}{rgb}{0.3333333333333333, 0.6588235294117647, 0.40784313725490196}

\title{Multilevel Monte Carlo for Polytopic DG Discretisations with Spatially Represented Random Coefficients}

\author[1]{Veronika Chronholm}

\author[2]{Matt Evans}

\author[3]{Emmanuil H. Georgoulis}

\author[1,2]{Tristan Pryer}
  
\address{$^1$ Institute for Mathematical Innovation\\ University of
  Bath, Bath, UK. $^2$ Department of Mathematical Sciences
  \\ University of Bath, Bath, UK.}
\address{  The Maxwell Institute for Mathematical Sciences \& Department of Mathematics, \\ Heriot-Watt University, Edinburgh EH14 4AS, UK, and \\ Department of Mathematics, School of Applied Mathematical and Physical Sciences, \\ National
  Technical University of Athens, Zografou 15780, Greece, and \\
  IACM-FORTH, Greece}
\email{e.georgoulis@hw.ac.uk}

\begin{document}

\begin{abstract}
  We develop a multilevel Monte Carlo framework for coercive elliptic
  problems with random diffusion and reaction coefficients on
  polytopic meshes, with particular emphasis on the approximation and
  inter-level coupling of random coefficients generated from spatially
  represented latent fields. The main analytical contribution is a
  Strang-type error decomposition in a DG norm that separates the
  forward discretisation error from the error introduced by
  approximating the random coefficients. On this basis, we study two
  coefficient-generation and multilevel-coupling strategies, a
  spectral sampler based on a stationary Gaussian latent field defined
  on a fixed Cartesian reference grid, and an SPDE-based Mat\'ern
  sampler in which levels are coupled through projected white-noise
  loads on agglomerated hierarchies. Both fit into the same
  coefficient-aware analytical framework and couple adjacent levels
  directly through a common latent random input. The theory is
  supported by numerical experiments on general Voronoi and
  agglomerated polytopic meshes, demonstrating the predicted forward
  and coefficient-representation errors, decay of multilevel
  corrections, strong inter-level coupling, and the resulting
  work--accuracy behaviour relative to single-level Monte Carlo.
\end{abstract}

\maketitle

\section{Introduction}\label{sec:intro}

Elliptic boundary value problems with heterogeneous diffusion and
reaction coefficients arise throughout scientific computing and
engineering. In many applications these coefficients are not known
exactly and are more naturally modelled as random fields. The
computational task then becomes the estimation of expectations and
other statistics of quantities of interest depending on the solution
of a parameter-dependent PDE. For problems of realistic size, the
dominant cost lies in the repeated numerical solution of the forward
model.

Standard Monte Carlo methods are robust but become prohibitively
expensive once fine spatial resolution is required. Multilevel Monte
Carlo (MLMC) methods address this by combining many inexpensive coarse
solves with far fewer expensive fine solves, and can reduce total work
to near that of a single deterministic solve at the target accuracy
\cite{Heinrich2001MLMC,Giles2008MLMC,Giles2015Acta,BarthSchwabZollinger2011NumerMath}.
For elliptic PDEs with random coefficients, MLMC theory has been
developed for conforming finite element discretisations on nested mesh
hierarchies, including pathwise discretisation bounds and complexity
results for uniformly elliptic and lognormal models
\cite{CliffeGilesScheichlTeckentrup2011,CharrierScheichlTeckentrup2013,
TeckentrupScheichlGilesUllmann2013}.

Two related issues arise when MLMC is combined with spatially
represented random fields on general meshes. The first is geometric,
the forward discretisations may be non-nested or may be constructed by
agglomeration. White-noise sampling and MLMC coupling across nested
and non-nested finite element hierarchies have been developed in
\cite{CrociGilesRognesFarrell2018}, while recent work has combined
MLMC with Virtual Element discretisations on general polytopal meshes
using agglomerated coarse spaces
\cite{AntoniettiBonizzoniPerugiaVerani2026MLMCVEM}. The second issue
is the representation of the random input itself. When a coefficient
is generated through a spatially discretised latent field, its
approximation becomes part of the overall discretisation error and its
inter-level coupling directly affects the variance of the multilevel
corrections. SPDE-based sampling of random-field coefficients has been
studied in the multilevel quasi-Monte Carlo setting
\cite{CrociGilesFarrell2021MLQMC}, including coupling on non-nested
finite element meshes. The present work focuses on this latter
coefficient-representation issue in the setting of polytopic DG
discretisations, where strong pairwise coupling remains essential for
MLMC efficiency \cite{ChadaHoelJasraZouraris2022StrongCoupling}.

Discontinuous Galerkin methods on polytopic meshes provide an
attractive discretisation route in such settings. They offer geometric
flexibility while retaining robust stability and high-order
approximation on general polytopic mesh families
\cite{CangianiGeorgoulisHouston2014,CangianiDongGeorgoulisHouston2016ADRPolytopic}.
Furthermore, the cardinality of local element bases is independent of
the element shape. More broadly, recent work on polytopic
discretisations has shown that this geometric setting supports a wider
family of numerical methods, including recovery-based formulations and
their extension to polytopic meshes
\cite{georgoulis2018recovered,dong2020recovered}. In a related
direction, recent composite and nodally bound-preserving constructions
underline that polytopic meshes can support stable and structure-aware
discretisations beyond the standard interior penalty framework
\cite{amiri2025nodally}. At the same time, agglomeration-based
multilevel constructions for DG on polytopic meshes have been
developed successfully for linear solver design, including multigrid
and Schwarz methods \cite{AntoniettiHoustonHuSartiVerani2017Multigrid,
  AntoniettiPennesi2019Vcycle,
  AntoniettiHoustonPennesiSuli2020Schwarz}.

The present work is motivated by the need to develop an MLMC framework
for random-coefficient elliptic problems on general polytopic DG meshes
in which level coupling is induced by the underlying random input,
rather than by remapping already discretised coefficient fields between
meshes. We consider coercive elliptic diffusion--reaction problems with
random diffusion and reaction coefficients, discretised by
lifting-based symmetric interior penalty DG methods on polytopic
meshes. The FFT construction can be used on unrelated forward meshes,
while the SPDE construction is developed on agglomerated hierarchies
with a natural coarse-to-fine inclusion.

Our first contribution is analytical. We derive a Strang-type
decomposition in the DG norm that separates the forward discretisation
error from the error introduced by approximating the random
coefficients. The formulation is written using a discrete lifting,
equivalently an $L^2$-projected numerical flux, so that coefficient
traces are avoided in the consistency analysis. This is the key
structural point for the multilevel analysis. Once the coefficients are
generated only approximately, whether by representation of a spectral
sample or by the numerical solution of an auxiliary stochastic PDE, the
total bias must distinguish approximation of the forward solution from
approximation of the data driving the forward problem.

Our second contribution is to instantiate this coefficient-aware
framework with two different coefficient-generation and coupling
strategies on polytopic DG meshes. The first is a spectral method in
which a stationary Gaussian latent field is
generated on a fixed Cartesian reference grid using FFT and
circulant-embedding ideas
\cite{WoodChan1994,DietrichNewsam1997,
GrahamKuoNuyensScheichlSloan2018CirculantEmbedding} and then evaluated
at the quadrature points required for DG assembly. In this way, all
levels are driven by the same represented latent-field realisation,
without transferring a discretised coefficient field between the
forward meshes. The second is an SPDE-based Mat\'ern construction in
which the latent field is defined through an auxiliary elliptic
stochastic PDE driven by spatial white noise
\cite{Whittle1954Plane,Matern1986,LindgrenRueLindstrom2011SPDE}. On an
agglomerated hierarchy, the levels are coupled through the driving
noise itself using the exact discrete relation between fine and coarse
white-noise loads. The SPDE route is also informed by DG analysis in
mesh-dependent norms for rough data, which provides a useful analytical
precedent for discretisations driven by irregular right-hand sides
\cite{georgoulis2017analysis}. Modern developments in the SPDE
viewpoint and in scalable Markov and graph representations provide
further context for this route
\cite{SanzAlonsoYang2022MaternGraph,
LindgrenBolinRue2022SPDEReview,BolinKirchner2020RationalSPDE}.

Our third contribution is to combine these ingredients in a single
coefficient-aware multilevel error analysis. For the forward problem
we obtain DG-norm bounds in which the influence of coefficient
approximation appears explicitly through $\leb q$-type norms of the
coefficient error. For the spectral sampler these terms are related to
the representation and interpolation error of the latent field. For
the SPDE sampler they are related to approximation properties of the
DG resolvent acting on spatial white noise. This yields a unified MLMC
framework in which both coefficient-generation strategies can be
analysed using the same forward DG technology.

The distinction from the closest existing multilevel constructions is
therefore in the role played by the random coefficient. White-noise
sampling and coupling for finite element SPDE discretisations, including
non-nested hierarchies, were developed in
\cite{CrociGilesRognesFarrell2018}, and this approach was subsequently
used for SPDE-generated random-field coefficients in a multilevel
quasi-Monte Carlo framework
\cite{CrociGilesFarrell2021MLQMC}. More recently,
\cite{AntoniettiBonizzoniPerugiaVerani2026MLMCVEM} developed an MLMC
Virtual Element method on general polytopal meshes with agglomerated
coarse spaces. In the latter setting the random diffusion coefficient
is treated as a given continuum realisation and the multilevel analysis
concerns the approximation of the corresponding forward problem.

Here the spatial representation of the random coefficient is included
explicitly in the multilevel discretisation. The lifting-based DG
formulation gives a coefficient-aware Strang decomposition separating
forward spatial approximation, projected-flux consistency and
coefficient-representation error without requiring traces of the
represented coefficient. This permits two different input-space
couplings to be treated within the same analysis. In the FFT
construction, all forward levels are driven by a common represented
latent field and the forward meshes need not be nested. In the SPDE
construction, the agglomerated DG spaces are naturally nested and the
coarse white-noise load is obtained by exact restriction of the
fine-level load. Thus the coefficient representation controls a
separate contribution to the bias, while the coupling of the latent
random input controls the strength of the multilevel corrections.

The numerical experiments complement the analysis on general Voronoi
and agglomerated polytopic meshes. They verify the expected spatial
rates of the $\mathbb P_1$ DG discretisation, isolate the
coefficient-representation contribution predicted by the Strang
decomposition, and measure the weak and variance decay of the
multilevel corrections. The coupling diagnostics show strong
correlation between neighbouring levels for both random-field
constructions. The final work experiments compare complete MLMC and
single-level Monte Carlo calculations and illustrate how the observed
variance decay and sample costs determine the practical multilevel
gain.

The remainder of the paper is organised as follows.
Section~\ref{sec:problem} introduces the random elliptic
diffusion--reaction problem, the admissible random-coefficient models
and the quantities of interest. Section~\ref{sec:polytopic_dg}
formulates the lifting-based polytopic DG discretisation and derives
the Strang-type decomposition separating spatial discretisation,
projected-flux and coefficient-representation errors.
Section~\ref{sec:sampling} develops the two random-field constructions:
the common-reference FFT sampler and the SPDE Mat\'ern sampler with
white-noise coupling on agglomerated hierarchies.
Section~\ref{sec:mlmc} combines these approximation estimates with the
MLMC bias--variance decomposition, derives the corresponding error and
work estimates, and describes the optimal sample allocation.
Finally, Section~\ref{sec:numerics} presents numerical experiments on
Voronoi and agglomerated polytopic meshes, verifying the forward and
coefficient-representation errors, the strength of the multilevel
couplings, and the resulting work--accuracy behaviour.

\section{Problem setup}
\label{sec:problem}

Throughout, $L^p(\Omega)$ and $H^s(\Omega)$ denote the usual Lebesgue
and Sobolev spaces, and $(\Xi,\mathcal F,\mathbb P)$ is a complete
probability space. Bochner spaces are denoted by $L^p(\Xi;X)$, with
expectation $\mathbb E$ and variance $\mathbb V$ in the usual sense.

\subsection{Elliptic model with random coefficients}
\label{subsec:elliptic_random}

Let
\begin{equation}
  V:=\hoz(\W).
  \label{eq:space_V}
\end{equation}
We model the diffusion and reaction coefficients as random fields
\begin{equation}
  A:\W\times\Xi\to\mathbb R^{d\times d}_{\mathrm{sym}},
  \qquad
  c:\W\times\Xi\to\mathbb R,
  \label{eq:random_coeff_fields}
\end{equation}
measurable in $\xi\in\Xi$ and essentially bounded in $x\in\W$.  We
assume that there exist deterministic constants
\begin{equation}
  0<\underline a\le\overline a<\infty,
  \qquad
  0<\underline c\le\overline c<\infty,
  \label{eq:coeff_constants}
\end{equation}
such that
\begin{equation}
  \underline a\norm{\eta}^2
  \le
  \eta^\top A(x,\xi)\eta
  \le
  \overline a\norm{\eta}^2,
  \qquad
  \underline c
  \le
  c(x,\xi)
  \le
  \overline c,
  \label{eq:coeff_bounds}
\end{equation}
for almost every $x\in\W$, every $\eta\in\mathbb R^d$, and almost
every $\xi\in\Xi$.

Given $F\in\leb2(\W)$, for almost every $\xi\in\Xi$ we consider the
problem: find $u(\xi)\in V$ such that
\begin{equation}
  \acont{A(\xi)}{c(\xi)}{u(\xi)}{v}
  =
  \ltwop{F}{v}
  \qquad
  \forall v\in V,
  \label{eq:continuum_problem_random}
\end{equation}
where
\begin{equation}
  \acont{A}{c}{w}{v}
  :=
  \int_\W
  A(x)\nabla w\cdot\nabla v
  \d x
  +
  \int_\W
  c(x)wv
  \d x.
  \label{eq:continuum_bilinear_form}
\end{equation}
By \eqref{eq:coeff_bounds},
$\acont{A(\xi)}{c(\xi)}{\cdot}{\cdot}$ is uniformly coercive,
\begin{equation}
  \acont{A(\xi)}{c(\xi)}{v}{v}
  \ge
  \min\qc{\underline a,\underline c}
  \Norm{v}_{\sobh1(\W)}^2
  \qquad
  \forall v\in V,
  \label{eq:continuum_coercive_continuous}
\end{equation}
and uniformly continuous,
\begin{equation}
  \norm{
    \acont{A(\xi)}{c(\xi)}{w}{v}
  }
  \le
  \max\qc{\overline a,\overline c}
  \Norm{w}_{\sobh1(\W)}
  \Norm{v}_{\sobh1(\W)}
  \qquad
  \forall w,v\in V.
  \label{eq:continuum_continuity}
\end{equation}
Hence, by the Lax--Milgram theorem,
\eqref{eq:continuum_problem_random} admits a unique solution for almost
every $\xi\in\Xi$.  Moreover, there exists a constant
$C_{\mathrm{stab}}>0$, depending only on
$\underline a$, $\overline a$, $\underline c$, $\overline c$ and
$\W$, such that
\begin{equation}
  \Norm{u(\xi)}_{\sobh1(\W)}
  \le
  C_{\mathrm{stab}}
  \Norm{F}_{\leb2(\W)}
  \qquad
  \text{for a.e. }\xi\in\Xi.
  \label{eq:continuum_stability}
\end{equation}

\begin{remark}[Almost sure well-posedness and moment bounds]
  \label{rem:as_wellposed}
  Under the stated measurability assumptions, the solution map
  $u:\Xi\to V$ is strongly measurable.  Since
  \eqref{eq:continuum_stability} is uniform in $\xi$,
  \begin{equation}
    u
    \in
    \leb{\infty}\qp{\Xi;V},
    \qquad
    \Norm{u}_{\leb{\infty}\qp{\Xi;V}}
    \le
    C_{\mathrm{stab}}
    \Norm{F}_{\leb2(\W)}.
    \label{eq:uniform_moment_bound_u}
  \end{equation}
  Consequently,
  $u\in\leb p\qp{\Xi;V}$ for every $p\in[1,\infty)$ and
  \begin{equation}
    \mathbb E\qb{
      \Norm{u}_{\sobh1(\W)}^p
    }
    <\infty
    \qquad
    \forall p\in[1,\infty).
    \label{eq:moment_bound_u}
  \end{equation}
\end{remark}

\subsection{Random input models}
\label{subsec:random_inputs}

The analysis below uses only the pathwise bounds
\eqref{eq:coeff_bounds}.  We therefore allow different constructions
of the random coefficients, provided these bounds hold uniformly,
either directly or after a deterministic truncation.  We record here
the models used subsequently in the analysis and numerical
experiments.

A convenient scalar model is obtained from truncated lognormal
fields.  Let $Y_a$ and $Y_c$ be mean-zero Gaussian fields, which may
be correlated, and let
$m_a,m_c\in\leb{\infty}(\W)$ be deterministic.  Define
\begin{equation}
  \Pi_{[-\beta,\beta]}(y)
  :=
  \min\qc{
    \beta,
    \max\qc{-\beta,y}
  }.
  \label{eq:truncation_operator}
\end{equation}
For truncation parameters $\beta_a,\beta_c>0$, set
\begin{equation}
  a(x,\xi)
  :=
  \exp\qp{
    m_a(x)
    +
    \Pi_{[-\beta_a,\beta_a]}
    Y_a(x,\xi)
  },
  \label{eq:lognormal_a}
\end{equation}
and
\begin{equation}
  c(x,\xi)
  :=
  \exp\qp{
    m_c(x)
    +
    \Pi_{[-\beta_c,\beta_c]}
    Y_c(x,\xi)
  }.
  \label{eq:lognormal_c}
\end{equation}
Then
\begin{equation}
  \underline a
  :=
  \exp\qp{
    \essinf_\W m_a-\beta_a
  },
  \qquad
  \overline a
  :=
  \exp\qp{
    \esssup_\W m_a+\beta_a
  },
  \label{eq:a_bounds_from_trunc}
\end{equation}
and
\begin{equation}
  \underline c
  :=
  \exp\qp{
    \essinf_\W m_c-\beta_c
  },
  \qquad
  \overline c
  :=
  \exp\qp{
    \esssup_\W m_c+\beta_c
  }.
  \label{eq:c_bounds_from_trunc}
\end{equation}
Thus the resulting coefficients satisfy
\eqref{eq:coeff_bounds} almost surely.

\begin{remark}[Scalar diffusion]
  \label{rem:scalar_diffusion}
  In the scalar case we take
  \begin{equation}
    A(x,\xi)
    =
    a(x,\xi)I.
    \label{eq:scalar_diffusion}
  \end{equation}
  The diffusion condition in \eqref{eq:coeff_bounds} then reduces to
  \[
    \underline a
    \le
    a(x,\xi)
    \le
    \overline a.
  \]
  This is the setting used in the numerical experiments below.  A
  single latent field may be used for both coefficients by taking
  $Y_a=Y_c$, while distinct or correlated fields are equally
  admissible for the analysis.
\end{remark}

\subsection{Quantities of interest}
\label{subsec:qoi}

We consider quantities of interest of the form
\begin{equation}
  Q:V\to\mathbb R,
  \qquad
  Q^\star(\xi)
  :=
  Q\qp{u(\xi)}.
  \label{eq:qoi_def}
\end{equation}
For the multilevel analysis we focus on functionals which are
Lipschitz continuous with respect to the $\sobh1(\W)$ norm.  The
stability and forward approximation estimates can then be transferred
directly to the bias and to the multilevel corrections.

A canonical example is a bounded linear functional
\begin{equation}
  Q(u)
  :=
  \ltwop{g}{u}_{V^\ast,V},
  \qquad
  g\in V^\ast.
  \label{eq:qoi_linear}
\end{equation}
Indeed,
\begin{equation}
  \norm{
    Q(w)-Q(z)
  }
  \le
  \Norm{g}_{V^\ast}
  \Norm{w-z}_{\sobh1(\W)}
  \qquad
  \forall w,z\in V.
  \label{eq:qoi_linear_lip}
\end{equation}
In particular, if $g\in\leb2(\W)$, then
\[
  Q(u)
  =
  (g,u)_{\leb2(\W)}.
\]

Another useful example is a regional average.  If
$\omega\subset\W$ is measurable with $|\omega|>0$, define
\begin{equation}
  Q(u)
  :=
  \frac{1}{|\omega|}
  \int_\omega
  u(x)
  \d x.
  \label{eq:qoi_average}
\end{equation}
Then
\begin{equation}
  \norm{
    Q(w)-Q(z)
  }
  \le
  |\omega|^{-1/2}
  \Norm{w-z}_{\leb2(\W)}
  \le
  |\omega|^{-1/2}
  \Norm{w-z}_{\sobh1(\W)}.
  \label{eq:qoi_average_lip}
\end{equation}
Since point evaluation is not a continuous functional on
$\sobh1(\W)$ when $d\ge2$, local averages of the form
\eqref{eq:qoi_average} provide a convenient replacement whenever a
spatially localised output is required.

\section{Polytopic DG discretisation and Strang decomposition}
\label{sec:polytopic_dg}

We formulate the polytopic DG method independently of the particular
sampling strategy used for the random coefficients.  On each level
$\ell$, the forward problem is assembled using the represented
coefficient pair
\begin{equation}
  \widetilde A_\ell
  \in
  \leb{\infty}\qp{\W;\mathbb R^{d\times d}_{\mathrm{sym}}},
  \qquad
  \widetilde c_\ell
  \in
  \leb{\infty}(\W).
  \label{eq:discrete_coeff_pair}
\end{equation}
These are the coefficient fields actually seen by the DG
discretisation.  Their construction is discussed in
Section~\ref{sec:sampling}.

Throughout this section we assume the uniform bounds
\begin{equation}
  \underline a\norm{\xi}^2
  \le
  \xi^\top\widetilde A_\ell(x)\xi
  \le
  \overline a\norm{\xi}^2,
  \qquad
  0<\underline c
  \le
  \widetilde c_\ell(x)
  \le
  \overline c,
  \label{eq:dg_discrete_coeff_bounds}
\end{equation}
for almost every $x\in\W$, every $\xi\in\mathbb R^d$ and every
$\ell=0,\ldots,L$, with constants independent of $\ell$.

\subsection{Polytopic meshes and broken spaces}
\label{subsec:polytopic_meshes}

Let $\qc{\mathcal T_\ell}_{\ell=0}^L$ be a family of partitions of
$\W$ into non-overlapping open polytopes $K$ such that
\begin{equation}
  \overline\W
  =
  \bigcup_{K\in\mathcal T_\ell}\overline K,
  \qquad
  K\cap K'
  =
  \emptyset
  \quad
  \text{for }K\neq K'.
  \label{eq:mesh_partition}
\end{equation}
For $K\in\mathcal T_\ell$, let
\[
  h_K
  :=
  \operatorname{diam}(K),
\]
and set
\begin{equation}
  h_\ell
  :=
  \max_{K\in\mathcal T_\ell}h_K.
  \label{eq:mesh_size_level}
\end{equation}

Let $\mathcal E_\ell$ denote the set of open $(d-1)$-dimensional
faces of $\mathcal T_\ell$, decomposed into interior and boundary
faces,
\begin{equation}
  \mathcal E_\ell^i
  :=
  \ensemble{
    e=\partial K^-\cap\partial K^+
  }{
    K^\pm\in\mathcal T_\ell,\ K^-\neq K^+
  },
  \qquad
  \mathcal E_\ell^b
  :=
  \ensemble{
    e\subset\partial K\cap\partial\W
  }{
    K\in\mathcal T_\ell
  }.
  \label{eq:face_sets}
\end{equation}
For each $e\in\mathcal E_\ell^i$, fix a unit normal $\vec n_e$
pointing from $K^-$ to $K^+$, while on
$e\in\mathcal E_\ell^b$ we take $\vec n_e$ to be the outward unit
normal.

The face scale is defined in terms of the neighbouring element
diameters rather than the face measure.  We set
\begin{equation}
  h_e
  :=
  \begin{cases}
    \displaystyle
    \frac{2h_{K^-}h_{K^+}}
         {h_{K^-}+h_{K^+}},
    &
    e=\partial K^-\cap\partial K^+
    \in\mathcal E_\ell^i,
    \\[3mm]
    h_K,
    &
    e\subset\partial K
    \in\mathcal E_\ell^b.
  \end{cases}
  \label{eq:face_scale}
\end{equation}
In particular, the stabilisation does not degenerate when a polytope
contains faces of small measure \cite{botti2025trace}.

We impose a uniform chunkiness condition on the elements.  Namely,
there exists $\vartheta_\ast\in(0,1]$, independent of $\ell$, such
that for every $K\in\mathcal T_\ell$ there exist
$x_K\in K$ and $\rho_K\ge\vartheta_\ast h_K$ satisfying
\begin{equation}
  B(x_K,\rho_K)
  \subset
  K,
  \label{eq:poly_shape_reg}
\end{equation}
and $K$ is star-shaped with respect to this ball, in the sense that
\[
  \qb{x,y}
  \subset
  K
  \qquad
  \forall x\in B(x_K,\rho_K),
  \quad
  \forall y\in K.
\]
No lower bound is imposed on the measure of an individual face, and
the number of faces of an element is not assumed to be uniformly
bounded.

For a scalar field $v$ defined elementwise, let $v^\pm$ denote its
traces on
$e=\partial K^-\cap\partial K^+\in\mathcal E_\ell^i$ and define
\begin{equation}
  \jump{v}
  :=
  v^--v^+,
  \qquad
  \avg{v}
  :=
  \frac12\qp{v^-+v^+}.
  \label{eq:jump_average}
\end{equation}
On boundary faces we set
\begin{equation}
  \jump{v}
  :=
  v,
  \qquad
  \avg{v}
  :=
  v.
  \label{eq:jump_average_boundary}
\end{equation}
For vector fields, averages are taken componentwise.

Fix a polynomial degree $p\ge1$.  On each level define
\begin{equation}
  V_\ell
  :=
  \ensemble{
    v\in\leb2(\W)
  }{
    v|_K\in\mathbb P_p(K)
    \text{ for all }K\in\mathcal T_\ell
  },
  \label{eq:dg_space}
\end{equation}
with broken gradient
\begin{equation}
  \qp{\nabla_\ell v}|_K
  :=
  \nabla\qp{v|_K}.
  \label{eq:broken_gradient}
\end{equation}
We also use
\begin{equation}
  \sobh1(\mathcal T_\ell)
  :=
  \ensemble{
    v\in\leb2(\W)
  }{
    v|_K\in\sobh1(K)
    \text{ for all }K\in\mathcal T_\ell
  }.
  \label{eq:broken_h1}
\end{equation}

\begin{lemma}[Trace and inverse inequalities]
  \label{lem:poly_trace_inverse}
  Assume \eqref{eq:poly_shape_reg}.  There exists
  $C_{\mathrm{tr}}>0$, depending only on $\vartheta_\ast$ and the
  dimension, such that
  \begin{equation}
    \Norm{v}_{\leb2(\partial K)}^2
    \le
    C_{\mathrm{tr}}
    \qb{
      h_K^{-1}
      \Norm{v}_{\leb2(K)}^2
      +
      h_K
      \Norm{\nabla v}_{\leb2(K)}^2
    }
    \qquad
    \forall v\in\sobh1(K).
    \label{eq:poly_trace}
  \end{equation}
  Moreover, for fixed $p\ge1$ there exists
  $C_{\mathrm{inv}}>0$, independent of $K$ and $\ell$, such that
  \begin{equation}
    \Norm{\nabla v}_{\leb2(K)}
    \le
    C_{\mathrm{inv}}
    \frac{p^2}{h_K}
    \Norm{v}_{\leb2(K)}
    \qquad
    \forall v\in\mathbb P_p(K).
    \label{eq:poly_inverse}
  \end{equation}
  Consequently,
  \begin{equation}
    \Norm{v}_{\leb2(\partial K)}
    \lesssim
    h_K^{-1/2}
    \Norm{v}_{\leb2(K)}
    \qquad
    \forall v\in\mathbb P_p(K),
    \label{eq:poly_trace_inverse}
  \end{equation}
  where the hidden constant may depend on $p$ but is independent of
  the number and relative sizes of the faces of $K$.
\end{lemma}

\subsection{Lifting operator and the DG form}
\label{subsec:dg_formulation}

Introduce the broken vector-polynomial space
\begin{equation}
  \Sigma_\ell
  :=
  \ensemble{
    \vec\tau\in\leb2(\W;\mathbb R^d)
  }{
    \vec\tau|_K\in\qb{\mathbb P_p(K)}^d
    \text{ for all }K\in\mathcal T_\ell
  }.
  \label{eq:lifting_space}
\end{equation}
The lifting
\[
  \mathcal R_\ell:
  \leb2(\mathcal E_\ell)
  \longrightarrow
  \Sigma_\ell
\]
is defined by
\begin{equation}
  \int_\W
  \mathcal R_\ell(\phi)\cdot\vec\tau
  \d x
  =
  -
  \sum_{e\in\mathcal E_\ell}
  \int_e
  \phi
  \avg{\vec\tau\cdot\vec n_e}
  \d s
  \qquad
  \forall\vec\tau\in\Sigma_\ell.
  \label{eq:lifting_def}
\end{equation}

\begin{lemma}[Stability of the lifting]
  \label{lem:lifting_stability}
  Under \eqref{eq:poly_shape_reg},
  \begin{equation}
    \Norm{\mathcal R_\ell(\phi)}_{\leb2(\W)}
    \lesssim
    \qp{
      \sum_{e\in\mathcal E_\ell}
      h_e^{-1}
      \Norm{\phi}_{\leb2(e)}^2
    }^{1/2},
    \label{eq:lifting_stability}
  \end{equation}
  with a constant independent of the number and relative sizes of the
  faces.
\end{lemma}

\begin{proof}
  By \eqref{eq:lifting_def}, weighted Cauchy--Schwarz gives
  \[
    \left|
      \int_\W
      \mathcal R_\ell(\phi)\cdot\vec\tau
      \d x
    \right|
    \le
    \qp{
      \sum_{e\in\mathcal E_\ell}
      h_e^{-1}
      \Norm{\phi}_{\leb2(e)}^2
    }^{1/2}
    \qp{
      \sum_{e\in\mathcal E_\ell}
      h_e
      \Norm{
        \avg{\vec\tau\cdot\vec n_e}
      }_{\leb2(e)}^2
    }^{1/2}.
  \]
  Since $h_e\lesssim h_K$ for every face of $K$, summing over the
  complete element boundaries and using
  \eqref{eq:poly_trace_inverse} yields
  \[
    \sum_{e\in\mathcal E_\ell}
    h_e
    \Norm{
      \avg{\vec\tau\cdot\vec n_e}
    }_{\leb2(e)}^2
    \lesssim
    \sum_{K\in\mathcal T_\ell}
    h_K
    \Norm{\vec\tau}_{\leb2(\partial K)}^2
    \lesssim
    \Norm{\vec\tau}_{\leb2(\W)}^2.
  \]
  Taking the supremum over
  $\vec\tau\in\Sigma_\ell\setminus\qc{0}$ proves the result.
\end{proof}

In particular,
\begin{equation}
  \Norm{
    \mathcal R_\ell\qp{\jump{v_\ell}}
  }_{\leb2(\W)}
  \lesssim
  \qp{
    \sum_{e\in\mathcal E_\ell}
    h_e^{-1}
    \Norm{\jump{v_\ell}}_{\leb2(e)}^2
  }^{1/2}.
  \label{eq:lifting_jump_bound}
\end{equation}

For a symmetric positive-definite coefficient field $B$, define the
elemental diffusion scale
\begin{equation}
  \alpha_K(B)
  :=
  \esssup_{x\in K}
  \lambda_{\max}\qp{B(x)}
  \label{eq:element_diffusion_scale}
\end{equation}
and the corresponding face scale
\begin{equation}
  \alpha_e(B)
  :=
  \begin{cases}
    \max\qc{
      \alpha_{K^-}(B),
      \alpha_{K^+}(B)
    },
    &
    e=\partial K^-\cap\partial K^+
    \in\mathcal E_\ell^i,
    \\[1mm]
    \alpha_K(B),
    &
    e\subset\partial K
    \in\mathcal E_\ell^b.
  \end{cases}
  \label{eq:face_diffusion_scale}
\end{equation}
For $\gamma>0$, set
\begin{equation}
  \pen{e}{B}
  :=
  \gamma
  \frac{p^2}{h_e}
  \alpha_e(B).
  \label{eq:penalty_def}
\end{equation}

Let
\[
  \mathbb V_\ell
  :=
  V+V_\ell.
\]
For $w,v\in\mathbb V_\ell$, define the lifting-based symmetric
interior-penalty form
\begin{equation}
  \begin{split}
    \adgdiff{\ell}{B}{w}{v}
    &:=
    \int_\W
    B\nabla_\ell w\cdot\nabla_\ell v
    \d x
    +
    \int_\W
    B\nabla_\ell w
    \cdot
    \mathcal R_\ell\qp{\jump{v}}
    \d x
    \\
    &\qquad
    +
    \int_\W
    B\mathcal R_\ell\qp{\jump{w}}
    \cdot
    \nabla_\ell v
    \d x
    +
    \sum_{e\in\mathcal E_\ell}
    \int_e
    \pen{e}{B}
    \jump{w}\jump{v}
    \d s.
  \end{split}
  \label{eq:dg_diffusion_form}
\end{equation}
The full diffusion--reaction form is
\begin{equation}
  \adg{\ell}{B}{d}{w}{v}
  :=
  \adgdiff{\ell}{B}{w}{v}
  +
  \int_\W
  d\,wv
  \d x.
  \label{eq:dg_full_form}
\end{equation}

To relate \eqref{eq:dg_diffusion_form} to the usual face-flux
formulation, let
\[
  \Pi_\ell:
  \leb2(\W;\mathbb R^d)
  \longrightarrow
  \Sigma_\ell
\]
denote the elementwise $\leb2$-orthogonal projector.  Since
$\mathcal R_\ell\qp{\jump{v}}\in\Sigma_\ell$,
\begin{equation}
  \begin{split}
    \int_\W
    B\nabla_\ell w
    \cdot
    \mathcal R_\ell\qp{\jump{v}}
    \d x
    &=
    \int_\W
    \Pi_\ell\qp{B\nabla_\ell w}
    \cdot
    \mathcal R_\ell\qp{\jump{v}}
    \d x
    \\
    &=
    -
    \sum_{e\in\mathcal E_\ell}
    \int_e
    \avg{
      \Pi_\ell\qp{B\nabla_\ell w}
      \cdot\vec n_e
    }
    \jump{v}
    \d s.
  \end{split}
  \label{eq:lifting_projected_flux}
\end{equation}
Thus \eqref{eq:dg_diffusion_form} is precisely a symmetric
interior-penalty form in which the numerical normal flux is the
$\leb2$ projection of the physical flux.  In particular, no trace of
the possibly rough coefficient $B$ is required in the definition of
the method.

The natural norm is
\begin{equation}
  \begin{split}
    \dgnorm{\ell}{B}{d}{v}^2
    &:=
    \Norm{
      B^{1/2}\nabla_\ell v
    }_{\leb2(\W)}^2
    +
    \Norm{
      d^{1/2}v
    }_{\leb2(\W)}^2
    +
    \sum_{e\in\mathcal E_\ell}
    \Norm{
      \pen{e}{B}^{1/2}
      \jump{v}
    }_{\leb2(e)}^2.
  \end{split}
  \label{eq:dg_norm_def}
\end{equation}

\begin{theorem}[Coercivity and continuity]
  \label{thm:dg_coercive_continuous}
  Assume \eqref{eq:poly_shape_reg} and
  \eqref{eq:dg_discrete_coeff_bounds}.  There exists $\gamma_0>0$
  such that, for every $\gamma\ge\gamma_0$,
  \begin{equation}
    \adg{\ell}
         {\widetilde A_\ell}
         {\widetilde c_\ell}
         {v_\ell}
         {v_\ell}
    \ge
    C_{\mathrm{coer}}
    \dgnorm{\ell}
           {\widetilde A_\ell}
           {\widetilde c_\ell}
           {v_\ell}^2
    \qquad
    \forall v_\ell\in V_\ell,
    \label{eq:dg_coercive}
  \end{equation}
  and
  \begin{equation}
    \left|
      \adg{\ell}
           {\widetilde A_\ell}
           {\widetilde c_\ell}
           {w}
           {v}
    \right|
    \le
    C_{\mathrm{cont}}
    \dgnorm{\ell}
           {\widetilde A_\ell}
           {\widetilde c_\ell}
           {w}
    \dgnorm{\ell}
           {\widetilde A_\ell}
           {\widetilde c_\ell}
           {v}
    \label{eq:dg_continuity}
  \end{equation}
  for all $w,v\in\mathbb V_\ell$.  The constants are independent of
  $\ell$ and of the number and relative sizes of the faces.
\end{theorem}

\begin{proof}
  The volume and reaction terms are controlled directly by
  \eqref{eq:dg_norm_def}.  For the lifting terms,
  \eqref{eq:lifting_stability} and the uniform coefficient bounds give
  \[
    \left|
      \int_\W
      \widetilde A_\ell\nabla_\ell w
      \cdot
      \mathcal R_\ell\qp{\jump{v}}
      \d x
    \right|
    \lesssim
    \dgnorm{\ell}
           {\widetilde A_\ell}
           {\widetilde c_\ell}
           {w}
    \dgnorm{\ell}
           {\widetilde A_\ell}
           {\widetilde c_\ell}
           {v}.
  \]
  The same estimate holds with $w$ and $v$ interchanged, proving
  continuity.  Taking $w=v=v_\ell$ and applying Young's inequality
  to the two identical lifting terms shows that they can be absorbed
  by the volume and penalty terms once $\gamma$ is sufficiently
  large, which proves coercivity.
\end{proof}

Given $L_\ell\in V_\ell'$, find $U_\ell\in V_\ell$ such that
\begin{equation}
  \adg{\ell}
       {\widetilde A_\ell}
       {\widetilde c_\ell}
       {U_\ell}
       {v_\ell}
  =
  L_\ell(v_\ell)
  \qquad
  \forall v_\ell\in V_\ell.
  \label{eq:dg_problem_general_rhs}
\end{equation}
Equip $V_\ell'$ with
\begin{equation}
  \Norm{L_\ell}_{V_\ell'}
  :=
  \sup_{v_\ell\in V_\ell\setminus\qc{0}}
  \frac{
    \norm{L_\ell(v_\ell)}
  }{
    \dgnorm{\ell}
           {\widetilde A_\ell}
           {\widetilde c_\ell}
           {v_\ell}
  }.
  \label{eq:dual_norm_discrete}
\end{equation}

\begin{corollary}[Discrete well-posedness and stability]
  \label{cor:dg_discrete_wellposed}
  Under the hypotheses of
  Theorem~\ref{thm:dg_coercive_continuous},
  \eqref{eq:dg_problem_general_rhs} has a unique solution and
  \begin{equation}
    \dgnorm{\ell}
           {\widetilde A_\ell}
           {\widetilde c_\ell}
           {U_\ell}
    \lesssim
    \Norm{L_\ell}_{V_\ell'}.
    \label{eq:dg_stability_general_rhs}
  \end{equation}
\end{corollary}

We therefore define
\begin{equation}
  \Sdisc{\ell}
         {\widetilde A_\ell}
         {\widetilde c_\ell}
  :
  V_\ell'
  \longrightarrow
  V_\ell,
  \qquad
  \Sdisc{\ell}
         {\widetilde A_\ell}
         {\widetilde c_\ell}
  L_\ell
  :=
  U_\ell.
  \label{eq:discrete_solution_operator}
\end{equation}
For the forward problem,
\begin{equation}
  F_\ell(v_\ell)
  :=
  \ltwop{F}{v_\ell}
  =
  \int_\W
  Fv_\ell
  \d x.
  \label{eq:forward_rhs_restriction}
\end{equation}
The same discrete solution operator will later be used for the latent
SPDE with a Gaussian load functional.

\subsection{Strang decomposition}
\label{subsec:dg_strang}

Let
\[
  A
  \in
  \leb{\infty}
  \qp{
    \W;
    \mathbb R^{d\times d}_{\mathrm{sym}}
  },
  \qquad
  c\in\leb{\infty}(\W),
\]
satisfy
\begin{equation}
  \underline a\norm{\xi}^2
  \le
  \xi^\top A(x)\xi
  \le
  \overline a\norm{\xi}^2,
  \qquad
  0<\underline c
  \le
  c(x)
  \le
  \overline c.
  \label{eq:continuum_coeff_elliptic}
\end{equation}
Let $u\in V$ solve
\begin{equation}
  \acont{A}{c}{u}{v}
  =
  \ltwop{F}{v}
  \qquad
  \forall v\in V,
  \label{eq:continuum_target_problem}
\end{equation}
or equivalently
\begin{equation}
  u
  =
  \Scont{A}{c}F.
  \label{eq:continuous_solution_operator}
\end{equation}
The discrete solution is
\begin{equation}
  u_\ell
  :=
  \Sdisc{\ell}
         {\widetilde A_\ell}
         {\widetilde c_\ell}
  F_\ell.
  \label{eq:forward_discrete_solution_operator}
\end{equation}

For $z\in\mathbb V_\ell$ and $w_\ell\in V_\ell$, define
\begin{equation}
  \mathfrak C_\ell(z;w_\ell)
  :=
  F_\ell(w_\ell)
  -
  \adg{\ell}
       {\widetilde A_\ell}
       {\widetilde c_\ell}
       {z}
       {w_\ell}.
  \label{eq:consistency_defect}
\end{equation}

\begin{theorem}[Abstract Strang lemma]
  \label{thm:strang_abstract}
  Under the hypotheses of
  Theorem~\ref{thm:dg_coercive_continuous},
  \begin{equation}
    \begin{split}
      \dgnorm{\ell}
             {\widetilde A_\ell}
             {\widetilde c_\ell}
             {u-u_\ell}
      \lesssim
      \inf_{v_\ell\in V_\ell}
      \Bigg\{
      &
      \dgnorm{\ell}
             {\widetilde A_\ell}
             {\widetilde c_\ell}
             {u-v_\ell}
             +
      \sup_{w_\ell\in V_\ell\setminus\qc{0}}
      \frac{
        \norm{
          \mathfrak C_\ell(v_\ell;w_\ell)
        }
      }{
        \dgnorm{\ell}
               {\widetilde A_\ell}
               {\widetilde c_\ell}
               {w_\ell}
      }
      \Bigg\}.
    \end{split}
    \label{eq:strang_abstract}
  \end{equation}
\end{theorem}

There are now two distinct consistency effects.  The first is the
small inconsistency introduced by the lifted, equivalently projected,
numerical flux.  Define
\begin{equation}
  \mathfrak D_\ell(u)
  :=
  \sup_{w_\ell\in V_\ell\setminus\qc{0}}
  \frac{
    \left|
      F_\ell(w_\ell)
      -
      \adg{\ell}{A}{c}{u}{w_\ell}
    \right|
  }{
    \dgnorm{\ell}
           {\widetilde A_\ell}
           {\widetilde c_\ell}
           {w_\ell}
  }.
  \label{eq:discretisation_consistency_defect}
\end{equation}
The second is the error caused by replacing $(A,c)$ by the represented
pair $(\widetilde A_\ell,\widetilde c_\ell)$:
\begin{equation}
  \mathfrak E_\ell(u)
  :=
  \sup_{w_\ell\in V_\ell\setminus\qc{0}}
  \frac{
    \left|
      \adg{\ell}{A}{c}{u}{w_\ell}
      -
      \adg{\ell}
           {\widetilde A_\ell}
           {\widetilde c_\ell}
           {u}{w_\ell}
    \right|
  }{
    \dgnorm{\ell}
           {\widetilde A_\ell}
           {\widetilde c_\ell}
           {w_\ell}
  }.
  \label{eq:representation_defect}
\end{equation}

\begin{theorem}[Strang decomposition with representation defect]
  \label{thm:strang_representation}
  Under the preceding assumptions,
  \begin{equation}
    \begin{split}
      \dgnorm{\ell}
             {\widetilde A_\ell}
             {\widetilde c_\ell}
             {u-u_\ell}
      \lesssim
      &
      \inf_{v_\ell\in V_\ell}
      \dgnorm{\ell}
             {\widetilde A_\ell}
             {\widetilde c_\ell}
             {u-v_\ell}
             +
      \mathfrak D_\ell(u)
      +
      \mathfrak E_\ell(u).
    \end{split}
    \label{eq:strang_representation}
  \end{equation}
\end{theorem}

\begin{proof}
  For arbitrary $v_\ell,w_\ell\in V_\ell$,
  \begin{align*}
    \mathfrak C_\ell(v_\ell;w_\ell)
    &=
    \qb{
      F_\ell(w_\ell)
      -
      \adg{\ell}{A}{c}{u}{w_\ell}
    }
    \\
    &\quad
    +
    \qb{
      \adg{\ell}{A}{c}{u}{w_\ell}
      -
      \adg{\ell}
           {\widetilde A_\ell}
           {\widetilde c_\ell}
           {u}{w_\ell}
    }
    \\
    &\quad
    +
    \adg{\ell}
         {\widetilde A_\ell}
         {\widetilde c_\ell}
         {u-v_\ell}
         {w_\ell}.
  \end{align*}
  Apply \eqref{eq:dg_continuity}, divide by the DG norm of
  $w_\ell$, take the supremum, and substitute the result into
  \eqref{eq:strang_abstract}.
\end{proof}

\subsection{Approximation and lifting consistency}
\label{subsec:dg_consistency}

The regularity assumptions can be expressed directly in
$W^{1+s,r}$ rather than obtained indirectly from Sobolev embedding.
For $r\in[2,\infty]$ and $s>0$, write
\[
  \Norm{z}_{W^{s,r}(\mathcal T_\ell)}
  :=
  \qp{
    \sum_{K\in\mathcal T_\ell}
    \Norm{z}_{W^{s,r}(K)}^r
  }^{1/r},
\]
with the usual modification when $r=\infty$.

For $0<s\le p$ and $u\in W^{1+s,r}(\W)$, standard polynomial
approximation on uniformly star-shaped polytopes gives
\begin{equation}
  \inf_{v_\ell\in V_\ell}
  \dgnorm{\ell}
         {\widetilde A_\ell}
         {\widetilde c_\ell}
         {u-v_\ell}
  \lesssim
  h_\ell^s
  \Norm{u}_{W^{1+s,r}(\W)}.
  \label{eq:dg_approximation_term}
\end{equation}

The lifting consistency error is controlled by best approximation of
the physical flux $\vec q := A\nabla u$. Let $\Pi_\ell$ be the
elementwise $\leb2$ projector introduced above.

\begin{lemma}[Consistency of the projected flux]
  \label{lem:dg_consistency_exact}
  Let $r\in[2,\infty]$ and
  \[
    \frac1r<s\le p.
  \]
  Assume
  \[
    \vec q
    =
    A\nabla u
    \in
    W^{s,r}(\mathcal T_\ell).
  \]
  Then, for every $w_\ell\in V_\ell$,
  \begin{equation}
    \begin{split}
      F_\ell(w_\ell)
      -
      \adg{\ell}{A}{c}{u}{w_\ell}
      =
      -
      \sum_{e\in\mathcal E_\ell}
      \int_e
      \avg{
        \qp{
          \vec q-\Pi_\ell\vec q
        }
        \cdot\vec n_e
      }
      \jump{w_\ell}
      \d s.
    \end{split}
    \label{eq:dg_consistency_exact}
  \end{equation}
  Consequently,
  \begin{equation}
    \mathfrak D_\ell(u)
    \lesssim
    h_\ell^s
    \Norm{
      A\nabla u
    }_{W^{s,r}(\mathcal T_\ell)}.
    \label{eq:projected_flux_consistency_bound}
  \end{equation}
\end{lemma}

\begin{proof}
  Since $u\in H_0^1(\W)$, one has
  $\jump{u}=0$.  Elementwise integration by parts in
  \eqref{eq:continuum_target_problem} gives
  \[
    F_\ell(w_\ell)
    =
    \int_\W
    \vec q\cdot\nabla_\ell w_\ell
    \d x
    -
    \sum_{e\in\mathcal E_\ell}
    \int_e
    \avg{\vec q\cdot\vec n_e}
    \jump{w_\ell}
    \d s
    +
    \int_\W
    cuw_\ell
    \d x.
  \]
  On the other hand, by
  \eqref{eq:lifting_projected_flux},
  \[
    \adg{\ell}{A}{c}{u}{w_\ell}
    =
    \int_\W
    \vec q\cdot\nabla_\ell w_\ell
    \d x
    -
    \sum_{e\in\mathcal E_\ell}
    \int_e
    \avg{
      \Pi_\ell\vec q\cdot\vec n_e
    }
    \jump{w_\ell}
    \d s
    +
    \int_\W
    cuw_\ell
    \d x,
  \]
  which proves \eqref{eq:dg_consistency_exact}.

  Weighted Cauchy--Schwarz yields
  \begin{align*}
    &
    \left|
      F_\ell(w_\ell)
      -
      \adg{\ell}{A}{c}{u}{w_\ell}
    \right|
    \\
    &\qquad
    \le
    \qp{
      \sum_{e\in\mathcal E_\ell}
      h_e
      \Norm{
        \avg{
          \qp{
            \vec q-\Pi_\ell\vec q
          }
          \cdot\vec n_e
        }
      }_{\leb2(e)}^2
    }^{1/2}
    \qp{
      \sum_{e\in\mathcal E_\ell}
      h_e^{-1}
      \Norm{
        \jump{w_\ell}
      }_{\leb2(e)}^2
    }^{1/2}.
  \end{align*}
  The second factor is controlled by the DG norm.  Polynomial
  best approximation together with the trace estimate gives
  \[
    \qp{
      \sum_{e\in\mathcal E_\ell}
      h_e
      \Norm{
        \avg{
          \qp{
            \vec q-\Pi_\ell\vec q
          }
          \cdot\vec n_e
        }
      }_{\leb2(e)}^2
    }^{1/2}
    \lesssim
    h_\ell^s
    \Norm{\vec q}_{W^{s,r}(\mathcal T_\ell)}.
  \]
  Dividing by the DG norm of $w_\ell$ and taking the supremum proves
  the result.
\end{proof}

Combining \eqref{eq:dg_approximation_term} and
\eqref{eq:projected_flux_consistency_bound} gives
\begin{equation}
  \begin{split}
    \dgnorm{\ell}
           {\widetilde A_\ell}
           {\widetilde c_\ell}
           {u-u_\ell}
    \lesssim
    h_\ell^s
    \qb{
      \Norm{u}_{W^{1+s,r}(\W)}
      +
      \Norm{
        A\nabla u
      }_{W^{s,r}(\mathcal T_\ell)}
    }
    +
    \mathfrak E_\ell(u).
  \end{split}
  \label{eq:strang_with_h}
\end{equation}

\subsection{Coefficient representation defect}
\label{subsec:representation_defect}

The representation defect has a particularly simple form in the
lifting formulation.  Since $\jump{u}=0$,
\begin{equation}
  \begin{split}
    &
    \adg{\ell}{A}{c}{u}{w_\ell}
    -
    \adg{\ell}
         {\widetilde A_\ell}
         {\widetilde c_\ell}
         {u}{w_\ell}
    \\
    &\qquad
    =
    \int_\W
    \qp{
      A-\widetilde A_\ell
    }
    \nabla u
    \cdot
    \qb{
      \nabla_\ell w_\ell
      +
      \mathcal R_\ell\qp{\jump{w_\ell}}
    }
    \d x
    \\
    &\qquad\qquad
    +
    \int_\W
    \qp{
      c-\widetilde c_\ell
    }
    u w_\ell
    \d x.
  \end{split}
  \label{eq:defect_lifting_identity}
\end{equation}
Thus
\begin{equation}
  \begin{split}
    \mathfrak E_\ell(u)
    =
    \sup_{w_\ell\in V_\ell\setminus\qc{0}}
    \frac{1}{
      \dgnorm{\ell}
             {\widetilde A_\ell}
             {\widetilde c_\ell}
             {w_\ell}
    }
    \Bigg|    
    \int_\W
    \qp{
      A-\widetilde A_\ell
    }
    \nabla u
    \cdot
    \qb{
      \nabla_\ell w_\ell
      +
      \mathcal R_\ell\qp{\jump{w_\ell}}
    }
    \d x
    +
    \int_\W
    \qp{
      c-\widetilde c_\ell
    }
    u w_\ell
    \d x
    \Bigg|.
  \end{split}
  \label{eq:representation_defect_lifted}
\end{equation}

\begin{lemma}[$\leb q$ representation estimate]
  \label{lem:representation_defect_lq}
  Let $q\in[2,\infty]$ and define $r\in[2,\infty]$ by
  \begin{equation}
    \frac1q
    +
    \frac1r
    +
    \frac12
    =
    1.
    \label{eq:q_r_relation}
  \end{equation}
  Assume
  \[
    u\in W^{1,r}(\W),
    \qquad
    A-\widetilde A_\ell
    \in
    \leb q\qp{\W;\mathbb R^{d\times d}},
    \qquad
    c-\widetilde c_\ell
    \in
    \leb q(\W).
  \]
  Then
  \begin{equation}
    \begin{split}
      \mathfrak E_\ell(u)
      \lesssim
      &
      \Norm{
        A-\widetilde A_\ell
      }_{\leb q(\W)}
      \Norm{\nabla u}_{\leb r(\W)}    
      +
      \Norm{
        c-\widetilde c_\ell
      }_{\leb q(\W)}
      \Norm{u}_{\leb r(\W)}.
    \end{split}
    \label{eq:representation_defect_lq_bound}
  \end{equation}
\end{lemma}

\begin{proof}
  From \eqref{eq:defect_lifting_identity} and H\"older's inequality,
  \begin{align*}
    &
    \left|
      \adg{\ell}{A}{c}{u}{w_\ell}
      -
      \adg{\ell}
           {\widetilde A_\ell}
           {\widetilde c_\ell}
           {u}{w_\ell}
    \right|
    \\
    &\qquad
    \le
    \Norm{
      A-\widetilde A_\ell
    }_{\leb q(\W)}
    \Norm{\nabla u}_{\leb r(\W)}
    \Norm{
      \nabla_\ell w_\ell
      +
      \mathcal R_\ell\qp{\jump{w_\ell}}
    }_{\leb2(\W)}
    \\
    &\qquad\qquad
    +
    \Norm{
      c-\widetilde c_\ell
    }_{\leb q(\W)}
    \Norm{u}_{\leb r(\W)}
    \Norm{w_\ell}_{\leb2(\W)}.
  \end{align*}
  By \eqref{eq:lifting_jump_bound},
  \eqref{eq:dg_discrete_coeff_bounds} and
  \eqref{eq:dg_norm_def},
  \begin{equation}
    \Norm{
      \nabla_\ell w_\ell
      +
      \mathcal R_\ell\qp{\jump{w_\ell}}
    }_{\leb2(\W)}
    +
    \Norm{w_\ell}_{\leb2(\W)}
    \lesssim
    \dgnorm{\ell}
           {\widetilde A_\ell}
           {\widetilde c_\ell}
           {w_\ell}.
  \end{equation}
  Division by the DG norm and taking the supremum gives
  \eqref{eq:representation_defect_lq_bound}.
\end{proof}

We may therefore impose the $W^{1+s,r}$ regularity required by the
spatial approximation directly.  If
\[
  u\in W^{1+s,r}(\W),
\]
then
\begin{equation}
  \mathfrak E_\ell(u)
  \lesssim
  \qb{
    \Norm{
      A-\widetilde A_\ell
    }_{\leb q(\W)}
    +
    \Norm{
      c-\widetilde c_\ell
    }_{\leb q(\W)}
  }
  \Norm{u}_{W^{1+s,r}(\W)}.
  \label{eq:representation_defect_h1s}
\end{equation}

Combining
\eqref{eq:strang_with_h} and
\eqref{eq:representation_defect_h1s} yields
\begin{equation}
  \begin{split}
    \dgnorm{\ell}
           {\widetilde A_\ell}
           {\widetilde c_\ell}
           {u-u_\ell}
    &\lesssim
    h_\ell^s
    \qb{
      \Norm{u}_{W^{1+s,r}(\W)}
      +
      \Norm{
        A\nabla u
      }_{W^{s,r}(\mathcal T_\ell)}
    }
    +
    \qb{
      \Norm{
        A-\widetilde A_\ell
      }_{\leb q(\W)}
      +
      \Norm{
        c-\widetilde c_\ell
      }_{\leb q(\W)}
    }
    \Norm{u}_{W^{1+s,r}(\W)}.
  \end{split}
  \label{eq:strang_lq_final}
\end{equation}

It is convenient to introduce the pathwise regularity quantity
\begin{equation}
  \mathcal M_{s,r}(u;A)
  :=
  \Norm{u}_{W^{1+s,r}(\W)}
  +
  \Norm{
    A\nabla u
  }_{W^{s,r}(\mathcal T_\ell)}.
  \label{eq:forward_regularity_quantity}
\end{equation}
Then \eqref{eq:strang_lq_final} may be written compactly as
\begin{equation}
  \begin{split}
    \dgnorm{\ell}
           {\widetilde A_\ell}
           {\widetilde c_\ell}
           {u-u_\ell}
    \lesssim
    \qb{
      h_\ell^s
      +
      \Norm{
        A-\widetilde A_\ell
      }_{\leb q(\W)}
      +
      \Norm{
        c-\widetilde c_\ell
      }_{\leb q(\W)}
    }
    \mathcal M_{s,r}(u;A).
  \end{split}
  \label{eq:strang_lq_compact}
\end{equation}

The first contribution in \eqref{eq:strang_lq_compact} contains both
the best-approximation error of the solution and the controlled
consistency error introduced by the projected numerical flux.  The
remaining terms isolate the error caused by representing the random
coefficients on level $\ell$.  This is the decomposition used in the
multilevel analysis below.

\section{Sampling strategies for the random coefficients}
\label{sec:sampling}

This section describes two constructions of the represented
coefficient pair $\qp{\widetilde A_\ell,\widetilde c_\ell}$ appearing
in the forward DG problem.  Both constructions begin with a Gaussian
latent field and then apply deterministic coefficient maps.  They
differ in the object used to couple adjacent MLMC levels.

In the spectral construction, a single Gaussian field is represented
on a Cartesian reference grid and the same realised field is evaluated
on every polytopic forward mesh.  In the SPDE construction, the latent
field is approximated on an agglomerated polytopic hierarchy and
adjacent levels are coupled through consistent restrictions of the
same spatial white noise.

These are alternative multilevel couplings and are not coupled
pathwise to one another.  When they are compared numerically, their
correlation length, bulk variance and coefficient maps are chosen
consistently.  The periodic FFT model and the bounded-domain SPDE with
homogeneous Dirichlet conditions nevertheless define different
Gaussian measures near the boundary.

Throughout this section let
\begin{equation}
  \Phi_A:\W\times\mathbb R\longrightarrow\mathbb R^{d\times d}_{\mathrm{sym}},
  \qquad
  \Phi_c:\W\times\mathbb R\longrightarrow\mathbb R
  \label{eq:coefficient_maps}
\end{equation}
be measurable coefficient maps such that
\begin{equation}
  A(x,\xi)=\Phi_A\qp{x,Y(x,\xi)},
  \qquad
  c(x,\xi)=\Phi_c\qp{x,Y(x,\xi)}.
  \label{eq:coefficient_maps_latent}
\end{equation}
The truncated lognormal models of
Section~\ref{subsec:random_inputs} are particular examples.  We assume
that, uniformly in $x\in\W$, the maps satisfy
\begin{equation}
  \Norm{\Phi_A\qp{x,y_1}-\Phi_A\qp{x,y_2}}\le L_A\norm{y_1-y_2},
  \qquad
  \norm{\Phi_c\qp{x,y_1}-\Phi_c\qp{x,y_2}}\le L_c\norm{y_1-y_2}.
  \label{eq:coeff_map_lipschitz}
\end{equation}
on the range relevant to the truncated latent field.  The truncations
are chosen so that the uniform bounds
\eqref{eq:dg_discrete_coeff_bounds} hold for every represented
coefficient field considered below.

\subsection{Spectral sampling on a Cartesian reference grid}
\label{subsec:fft_sampling}

Let $\mathcal B:=\qb{0,L_{\mathcal B}}^d$ be a Cartesian box
containing $\overline\W$, interpreted periodically.  For
$\vec n\in\mathbb Z^d$, set
\begin{equation}
  \vec k_{\vec n}:=\frac{2\pi}{L_{\mathcal B}}\vec n,
  \label{eq:fft_wavenumbers}
\end{equation}
and let
$\phi_{\vec n}(x):=L_{\mathcal B}^{-d/2}
\exp\qp{\mathrm i\vec k_{\vec n}\cdot x}$
denote the corresponding orthonormal Fourier basis of
$\leb2(\mathcal B)$.

We consider the real-valued, mean-zero periodic Gaussian field
\begin{equation}
  Y(x,\xi):=\sum_{\vec n\in\mathbb Z^d}\sqrt{S_{\vec n}}\,
  \zeta_{\vec n}(\xi)\phi_{\vec n}(x),
  \label{eq:fft_continuum_field}
\end{equation}
where
\begin{equation}
  S_{\vec n}:=
  c_{\nu,d,\mathcal B}\tau^2
  \qp{\kappa^2+\norm{\vec k_{\vec n}}^2}^{-(\nu+d/2)}.
  \label{eq:matern_spectrum}
\end{equation}
Here $\kappa>0$ determines the correlation length, $\nu>0$ controls
the regularity, $\tau>0$ determines the variance and
$c_{\nu,d,\mathcal B}>0$ is the normalisation constant for the
periodic model.  The Gaussian coefficients satisfy
$\mathbb E[\zeta_{\vec n}]=0$,
$\mathbb E\qb{\zeta_{\vec n}\overline{\zeta_{\vec m}}}
=\delta_{\vec n\vec m}$ and
$\zeta_{-\vec n}=\overline{\zeta_{\vec n}}$, with the zero mode real.
Since $\nu>0$, the series \eqref{eq:fft_continuum_field} converges in
$\leb2\qp{\Xi;\leb2(\mathcal B)}$.  Its restriction to $\W$ defines
the target latent field for the FFT construction.

Fix a reference resolution $N_{\mathrm{ref}}\in\mathbb N$ and set
\begin{equation}
  \delta_{\mathrm{ref}}:=\frac{L_{\mathcal B}}{N_{\mathrm{ref}}}.
  \label{eq:fft_reference_spacing}
\end{equation}
The Cartesian grid points are
\begin{equation}
  x_{\vec i}:=\delta_{\mathrm{ref}}\vec i,
  \qquad
  \vec i\in\qc{0,\ldots,N_{\mathrm{ref}}-1}^d.
  \label{eq:reference_grid}
\end{equation}
Let $\Lambda_{\mathrm{ref}}\subset\mathbb Z^d$ denote the Fourier
modes represented by the corresponding $N_{\mathrm{ref}}^d$-point
discrete Fourier transform, with the standard Nyquist convention, and
define
\begin{equation}
  P_{\mathrm{ref}}Y:=
  \sum_{\vec n\in\Lambda_{\mathrm{ref}}}
  \sqrt{S_{\vec n}}\,\zeta_{\vec n}\phi_{\vec n}.
  \label{eq:fft_fourier_truncation}
\end{equation}

The reference-grid sample is the restriction of this same truncated
realisation to the Cartesian grid,
\begin{equation}
  \qp{Y_{\mathrm{ref}}}_{\vec i}:=
  P_{\mathrm{ref}}Y\qp{x_{\vec i}}.
  \label{eq:fft_reference_grid_values}
\end{equation}
Equivalently, under the chosen discrete Fourier normalisation,
\begin{equation}
  Y_{\mathrm{ref}}
  =\mathcal F^{-1}\qp{\sqrt{S_{\vec n}}\,
  \widehat W_{\mathrm{ref},\vec n}},
  \qquad
  \widehat W_{\mathrm{ref}}=\mathcal F W_{\mathrm{ref}},
  \label{eq:fft_sampling_discrete}
\end{equation}
where $W_{\mathrm{ref}}$ is a real centred Gaussian white-noise vector
and its Fourier coefficients are identified with the resolved
coefficients in \eqref{eq:fft_fourier_truncation}.  Thus
\eqref{eq:fft_sampling_discrete} is the computational realisation of
the Fourier truncation of the same field $Y$.  The zero Fourier mode
is retained.

Let $\mathcal I_{\mathrm{ref}}$ denote the Cartesian interpolation
operator applied to the grid values
\eqref{eq:fft_reference_grid_values}; in the numerical experiments
this is bilinear interpolation in two dimensions.  Define
\begin{equation}
  \widetilde Y_{\mathrm{ref}}
  :=\mathcal I_{\mathrm{ref}}\qp{P_{\mathrm{ref}}Y}
  \qquad\text{on }\W.
  \label{eq:fft_interpolation}
\end{equation}

For a fixed MLMC hierarchy, the same realised function
$\widetilde Y_{\mathrm{ref}}(\cdot,\xi)$ drives every forward level:
\begin{equation}
  \widetilde Y_\ell:=\widetilde Y_{\mathrm{ref}},
  \qquad
  \ell=0,\ldots,L.
  \label{eq:fft_common_latent_field}
\end{equation}
The represented coefficients are
\begin{equation}
  \widetilde A_\ell(x,\xi)
  :=\Phi_A\qp{x,\widetilde Y_{\mathrm{ref}}(x,\xi)},
  \qquad
  \widetilde c_\ell(x,\xi)
  :=\Phi_c\qp{x,\widetilde Y_{\mathrm{ref}}(x,\xi)}.
  \label{eq:fft_represented_coefficients}
\end{equation}
Hence both forward solves in the coupled correction
$\delta Q_\ell:=Q_\ell-Q_{\ell-1}$ use the same represented
coefficient realisation.  No coefficient field is transferred between
the forward meshes.

By construction,
\begin{equation}
  Y-\widetilde Y_{\mathrm{ref}}
  =
  \qp{Y-P_{\mathrm{ref}}Y}
  +
  \qp{P_{\mathrm{ref}}Y
  -\mathcal I_{\mathrm{ref}}\qp{P_{\mathrm{ref}}Y}}
  \qquad\text{on }\W.
  \label{eq:fft_latent_error_split}
\end{equation}

\begin{lemma}[Propagation of latent-field error]
  \label{lem:fft_coeff_lipschitz}
  Under \eqref{eq:coeff_map_lipschitz}, for every
  $q\in[1,\infty]$,
  \begin{equation}
    \Norm{A-\widetilde A_\ell}_{\leb q(\W)}
    +
    \Norm{c-\widetilde c_\ell}_{\leb q(\W)}
    \lesssim
    \Norm{Y-\widetilde Y_{\mathrm{ref}}}_{\leb q(\W)}
    \label{eq:fft_coeff_lipschitz_bound}
  \end{equation}
  almost surely, with a constant independent of $\ell$ and
  $N_{\mathrm{ref}}$.
\end{lemma}

\begin{proof}
  The pointwise estimate follows from
  \eqref{eq:coefficient_maps_latent},
  \eqref{eq:fft_represented_coefficients} and
  \eqref{eq:coeff_map_lipschitz}.  Taking the $\leb q(\W)$ norm gives
  the result.
\end{proof}

\begin{lemma}[Reference-grid approximation]
  \label{lem:fft_latent_lq}
  Let $q\in[1,\infty]$ and suppose that, for some
  $\mu\in(0,2]$ appropriate to the interpolation operator,
  $P_{\mathrm{ref}}Y\in W^{\mu,q}(\mathcal B)$ almost surely.  Then
  \begin{equation}
    \Norm{Y-\widetilde Y_{\mathrm{ref}}}_{\leb q(\W)}
    \lesssim
    \Norm{Y-P_{\mathrm{ref}}Y}_{\leb q(\mathcal B)}
    +
    \delta_{\mathrm{ref}}^\mu
    \Norm{P_{\mathrm{ref}}Y}_{W^{\mu,q}(\mathcal B)}.
    \label{eq:fft_latent_lq_bound}
  \end{equation}
\end{lemma}

\begin{proof}
  Apply the triangle inequality to
  \eqref{eq:fft_latent_error_split} and the Cartesian interpolation
  estimate to the second term.
\end{proof}

Combining Lemmas~\ref{lem:fft_coeff_lipschitz} and
\ref{lem:fft_latent_lq} gives
\begin{equation}
  \Norm{A-\widetilde A_\ell}_{\leb q(\W)}
  +
  \Norm{c-\widetilde c_\ell}_{\leb q(\W)}
  \lesssim
  \Norm{Y-P_{\mathrm{ref}}Y}_{\leb q(\mathcal B)}
  +
  \delta_{\mathrm{ref}}^\mu
  \Norm{P_{\mathrm{ref}}Y}_{W^{\mu,q}(\mathcal B)}.
  \label{eq:fft_coeff_lq_bound}
\end{equation}

For finite $\varrho\ge1$, define
\begin{equation}
  \varepsilon_{\mathrm{ref}}
  :=
  \Norm{
    \Norm{Y-P_{\mathrm{ref}}Y}_{\leb q(\mathcal B)}
    +
    \delta_{\mathrm{ref}}^\mu
    \Norm{P_{\mathrm{ref}}Y}_{W^{\mu,q}(\mathcal B)}
  }_{\leb{2\varrho}(\Xi)}.
  \label{eq:fft_reference_error}
\end{equation}
Then
\begin{equation}
  \Norm{
    \Norm{A-\widetilde A_\ell}_{\leb q(\W)}
    +
    \Norm{c-\widetilde c_\ell}_{\leb q(\W)}
  }_{\leb{2\varrho}(\Xi)}
  \lesssim
  \varepsilon_{\mathrm{ref}},
  \qquad
  \ell=0,\ldots,L.
  \label{eq:fft_reference_error_moment}
\end{equation}

For a target hierarchy with finest scale $h_L$, the resolution
$N_{\mathrm{ref}}=N_{\mathrm{ref}}(L)$ is selected at the accuracy
required for that hierarchy.  The resulting reference realisation is
then shared by every forward level $0,\ldots,L$.

\subsection{SPDE sampling on an agglomerated hierarchy}
\label{subsec:spde_sampling}

We next construct the latent field through the elliptic SPDE
\begin{equation}
  \qp{\kappa^2-\Delta}Y=\tau\mathcal W
  \quad\text{in }\W,
  \qquad
  Y=0
  \quad\text{on }\partial\W.
  \label{eq:spde_matern}
\end{equation}
Let
$\mathcal L_{\mathrm D}:=\kappa^2-\Delta_{\mathrm D}$ and
$\mathcal G:=\mathcal L_{\mathrm D}^{-1}$.  Then
\begin{equation}
  Y=\tau\mathcal G\mathcal W.
  \label{eq:continuum_spde_operator}
\end{equation}

Spatial white noise is understood as an isonormal Gaussian process
over $\leb2(\W)$:
\begin{equation}
  \mathbb E\qb{
    \ltwop{\mathcal W}{v}
    \ltwop{\mathcal W}{w}
  }
  =
  \qp{v,w}_{\leb2(\W)}
  \qquad
  \forall v,w\in\leb2(\W).
  \label{eq:white_noise_covariance}
\end{equation}

In the whole-space or periodic setting,
\eqref{eq:spde_matern} has spectral density proportional to
$\qp{\kappa^2+\norm{\vec k}^2}^{-2}$.  It therefore agrees with the
Mat\'ern spectrum \eqref{eq:matern_spectrum} when
\begin{equation}
  \nu+\frac d2=2.
  \label{eq:second_order_matern_match}
\end{equation}
For other Mat\'ern smoothness parameters a fractional or higher-order
SPDE representation may be used.

The homogeneous Dirichlet condition modifies the covariance near
$\partial\W$ and, in particular, pins the latent field to zero on the
boundary.  It is part of the bounded-domain SPDE model and is not
intended to reproduce the periodic FFT covariance close to the
boundary.

Let $\qc{\mathcal T_\ell}_{\ell=0}^L$ be an agglomerated hierarchy
such that every coarse cell $K\in\mathcal T_{\ell-1}$ is a union of
cells of $\mathcal T_\ell$:
\begin{equation}
  K=\bigcup_{k\in\mathcal C_\ell(K)}k.
  \label{eq:agglomeration_relation}
\end{equation}
For fixed polynomial degree this gives the natural inclusion
\begin{equation}
  V_{\ell-1}\subset V_\ell,
  \label{eq:agglomerated_space_inclusion}
\end{equation}
since a polynomial on a coarse agglomerate restricts to a polynomial
of the same degree on every fine child.

The latent SPDE is discretised by the same lifting-based symmetric
interior-penalty construction as the forward problem, with diffusion
coefficient $I$ and reaction coefficient $\kappa^2$.  Thus
\begin{equation}
  a_{\ell,\mathrm D}^{\mathrm{lat}}(w,v)
  :=
  a_\ell^{\mathrm{DG}}\qp{w,v;I,\kappa^2}.
  \label{eq:latent_dg_form}
\end{equation}
In particular, the lifting and penalty terms include the boundary
faces and impose the homogeneous Dirichlet condition weakly.

Define the discrete latent resolvent
$\mathcal G_\ell:V_\ell'\longrightarrow V_\ell$ by
\begin{equation}
  a_{\ell,\mathrm D}^{\mathrm{lat}}
  \qp{\mathcal G_\ell L_\ell,v_\ell}
  =
  L_\ell(v_\ell)
  \qquad
  \forall v_\ell\in V_\ell.
  \label{eq:spde_discrete_resolvent}
\end{equation}

\subsection{Discrete white noise and inter-level coupling}
\label{subsubsec:spde_noise_coupling}

The canonical level-$\ell$ white-noise functional is the restriction
of the same continuum white noise to $V_\ell$:
\begin{equation}
  \eta_\ell(v_\ell):=\ltwop{\mathcal W}{v_\ell},
  \qquad
  v_\ell\in V_\ell.
  \label{eq:noise_restriction}
\end{equation}
It satisfies
\begin{equation}
  \mathbb E\qb{\eta_\ell(v_\ell)\eta_\ell(w_\ell)}
  =
  \qp{v_\ell,w_\ell}_{\leb2(\W)}.
  \label{eq:noise_cov_identity}
\end{equation}

Let $\qc{\varphi_i^{(\ell)}}_{i=1}^{n_\ell}$ be a basis of $V_\ell$
and define the mass matrix
\begin{equation}
  \qp{M_\ell}_{ij}:=
  \int_\W\varphi_i^{(\ell)}\varphi_j^{(\ell)}\d x.
  \label{eq:mass_matrix_level}
\end{equation}
The white-noise load vector
\begin{equation}
  \qp{b_\ell}_i:=\eta_\ell\qp{\varphi_i^{(\ell)}}
  \label{eq:white_noise_load_vector}
\end{equation}
therefore satisfies
\begin{equation}
  b_\ell\sim\mathcal N\qp{0,M_\ell}.
  \label{eq:discrete_white_noise}
\end{equation}
If $v_\ell=\sum_i\mathbf v_i\varphi_i^{(\ell)}$, then
\begin{equation}
  \eta_\ell(v_\ell)=b_\ell^\top\mathbf v.
  \label{eq:white_noise_functional_vector}
\end{equation}
Thus $b_\ell$ is the Gaussian load vector representing the action of
white noise on $V_\ell$; it is not the coefficient vector of an
$\leb2$ function.

Since the DG mass matrix is block diagonal by elements,
\eqref{eq:discrete_white_noise} may be sampled locally.  If
$M_K=L_KL_K^\top$, then
\begin{equation}
  b_K=L_Kz_K,
  \qquad
  z_K\sim\mathcal N\qp{0,I},
  \label{eq:local_white_noise_sampling}
\end{equation}
and the elemental load vectors are mutually independent.

Let
$\mathbf P_{\ell-1}^{\ell}:
\mathbb R^{n_{\ell-1}}\longrightarrow\mathbb R^{n_\ell}$
denote the coefficient prolongation associated with
\eqref{eq:agglomerated_space_inclusion}.  Thus if
$v_{\ell-1}\in V_{\ell-1}$ has coefficient vector $\mathbf v$, then
$\mathbf P_{\ell-1}^{\ell}\mathbf v$ represents the same function as
an element of $V_\ell$.  Consequently,
\begin{equation}
  M_{\ell-1}
  =
  \qp{\mathbf P_{\ell-1}^{\ell}}^\top
  M_\ell
  \mathbf P_{\ell-1}^{\ell}.
  \label{eq:mass_matrix_agglomeration_identity}
\end{equation}

Given a fine-level white-noise load $b_\ell$, define
\begin{equation}
  b_{\ell-1}:=
  \qp{\mathbf P_{\ell-1}^{\ell}}^\top b_\ell.
  \label{eq:noise_coupling}
\end{equation}
Then
\begin{equation}
  \operatorname{Cov}\qp{b_{\ell-1}}
  =
  \qp{\mathbf P_{\ell-1}^{\ell}}^\top
  M_\ell
  \mathbf P_{\ell-1}^{\ell}
  =
  M_{\ell-1},
  \label{eq:noise_coupling_covariance}
\end{equation}
so the coarse marginal has exactly the required white-noise law.
Moreover,
\begin{equation}
  \eta_{\ell-1}(v_{\ell-1})
  =
  \eta_\ell(v_{\ell-1})
  \qquad
  \forall v_{\ell-1}\in V_{\ell-1},
  \label{eq:noise_consistency_identity}
\end{equation}
where on the right-hand side $v_{\ell-1}$ is identified with its
prolongation to $V_\ell$.  Thus adjacent levels are restrictions of
one and the same isonormal Gaussian process.

The coupled latent fields are
\begin{equation}
  Y_\ell:=\tau\mathcal G_\ell\eta_\ell,
  \qquad
  Y_{\ell-1}:=\tau\mathcal G_{\ell-1}\eta_{\ell-1},
  \label{eq:coupled_spde_fields}
\end{equation}
and the represented forward coefficients are
\begin{equation}
  \widetilde A_\ell(x,\xi):=\Phi_A\qp{x,Y_\ell(x,\xi)},
  \qquad
  \widetilde c_\ell(x,\xi):=\Phi_c\qp{x,Y_\ell(x,\xi)}.
  \label{eq:spde_represented_coefficients}
\end{equation}

\subsection{Approximation of the SPDE latent field}
\label{subsec:spde_latent_approximation}

The strong approximation of elliptic SPDEs driven by spatial white
noise by polytopic symmetric interior-penalty methods is analysed in
\cite{pryer2026polytopic}.  The analysis uses the canonical restriction
of the same continuum white-noise functional to the DG space and
separates stochastic projection of the forcing from deterministic SIP
approximation of the elliptic resolvent.

We assume that the polytopic hierarchy is quasi-uniform, independently
of the chunkiness assumption used in the DG analysis: there exists
$c_{\mathrm{qu}}>0$, independent of $\ell$, such that
\begin{equation}
  c_{\mathrm{qu}}h_\ell\le h_K\le h_\ell
  \qquad
  \forall K\in\mathcal T_\ell.
  \label{eq:spde_quasi_uniformity}
\end{equation}
We also assume the standard $H^2$ regularity estimate for the
homogeneous Dirichlet reaction--diffusion problem,
\begin{equation}
  \Norm{\mathcal L_{\mathrm D}^{-1}f}_{\sobh2(\W)}
  \lesssim
  \Norm{f}_{\leb2(\W)}
  \qquad
  \forall f\in\leb2(\W).
  \label{eq:dirichlet_h2_regular}
\end{equation}

Under these assumptions,
\cite[Theorem~3.15]{pryer2026polytopic} gives the sharp strong estimate
\begin{equation}
  \Norm{Y-Y_\ell}_{\leb2\qp{\Xi;\leb2(\W)}}
  \lesssim
  h_\ell^{\alpha_Y},
  \qquad
  \alpha_Y:=2-\frac d2,
  \qquad
  d\in\qc{2,3}.
  \label{eq:spde_strong_rate}
\end{equation}
Since $Y-Y_\ell$ is a centred $\leb2(\W)$-valued Gaussian random
variable, equivalence of finite Gaussian moments yields, for every
finite $\varrho\ge1$,
\begin{equation}
  \Norm{Y-Y_\ell}_{\leb{\varrho}\qp{\Xi;\leb2(\W)}}
  \lesssim
  h_\ell^{\alpha_Y}.
  \label{eq:spde_strong_rate_moments}
\end{equation}
Thus $\alpha_Y=1$ in $d=2$ and $\alpha_Y=\frac12$ in $d=3$.

\begin{lemma}[Propagation to the SPDE coefficients]
  \label{lem:spde_coeff_from_latent}
  Assume \eqref{eq:coeff_map_lipschitz} and the uniform deterministic
  bounds on the exact and represented coefficients.  Then, for every
  finite $\varrho\ge1$,
  \begin{equation}
    \Norm{A-\widetilde A_\ell}_{\leb{\varrho}\qp{\Xi;\leb2(\W)}}
    +
    \Norm{c-\widetilde c_\ell}_{\leb{\varrho}\qp{\Xi;\leb2(\W)}}
    \lesssim
    h_\ell^{2-d/2}.
    \label{eq:spde_coeff_l2_bound}
  \end{equation}
  Moreover, for every $q\in[2,\infty)$,
  \begin{equation}
    \Norm{A-\widetilde A_\ell}_{\leb{\varrho}\qp{\Xi;\leb q(\W)}}
    +
    \Norm{c-\widetilde c_\ell}_{\leb{\varrho}\qp{\Xi;\leb q(\W)}}
    \lesssim
    h_\ell^{\alpha_{\mathrm{spde}}},
    \qquad
    \alpha_{\mathrm{spde}}:=\frac{4-d}{q}.
    \label{eq:spde_coeff_lq_bound}
  \end{equation}
\end{lemma}

\begin{proof}
  The $\leb2(\W)$ estimate follows from
  \eqref{eq:coeff_map_lipschitz} and
  \eqref{eq:spde_strong_rate_moments}.  For $q\in[2,\infty)$,
  interpolation gives
  \begin{equation*}
    \Norm{v}_{\leb q(\W)}
    \le
    \Norm{v}_{\leb2(\W)}^{2/q}
    \Norm{v}_{\leb\infty(\W)}^{1-2/q}.
  \end{equation*}
  The coefficient errors are uniformly bounded in
  $\leb\infty(\W)$ by the deterministic coefficient bounds.  Applying
  this inequality pathwise, taking moments and using
  \eqref{eq:spde_strong_rate_moments} gives
  $\alpha_{\mathrm{spde}}
  =\frac{2}{q}\qp{2-\frac d2}
  =\frac{4-d}{q}$.
\end{proof}

\begin{remark}[An admissible parameter choice]
  \label{rem:spde_admissible_parameters}
  The SPDE coefficient estimate requires $q<\infty$.  In two
  dimensions, for example, one may take
  $q=4$, $r=4$ and $\frac14<s\le p$.  Then
  $\alpha_{\mathrm{spde}}=\frac12$.  The forward regularity condition
  introduced below remains a separate assumption on the random
  elliptic solution.
\end{remark}

\subsection{Transfer to the forward problem}
\label{subsec:combined_error}

We now combine the sampling estimates with the lifting-based Strang
bound of Section~\ref{subsec:representation_defect}.  Let
$q,r\in[2,\infty]$ satisfy
\begin{equation}
  \frac1q+\frac1r+\frac12=1,
  \label{eq:combined_qr_condition}
\end{equation}
and let
\begin{equation}
  \frac1r<s\le p,
  \label{eq:combined_s_condition}
\end{equation}
where $p$ is the polynomial degree of the forward DG space.

To make the level dependence explicit, define
\begin{equation}
  \mathcal M_{s,r,\ell}\qp{u;A}
  :=
  \Norm{u}_{W^{1+s,r}(\W)}
  +
  \Norm{A\nabla u}_{W^{s,r}(\mathcal T_\ell)}
  \label{eq:level_forward_regularity_quantity}
\end{equation}
and set
\begin{equation}
  \mathscr M_{s,r}
  :=
  \sup_{0\le\ell\le L}
  \mathcal M_{s,r,\ell}\qp{u;A}.
  \label{eq:combined_regularity_quantity}
\end{equation}
We assume
\begin{equation}
  \mathscr M_{s,r}\in\leb{2\varrho}(\Xi)
  \label{eq:combined_u_regularity}
\end{equation}
for the probability exponent $\varrho\in[1,\infty)$ under
consideration.  This is a regularity assumption on the random forward
problem and is independent of the latent-field approximation estimates
above.

Define the level-$\ell$ coefficient error
\begin{equation}
  D_\ell
  :=
  \Norm{A-\widetilde A_\ell}_{\leb q(\W)}
  +
  \Norm{c-\widetilde c_\ell}_{\leb q(\W)}.
  \label{eq:coefficient_error_level}
\end{equation}

\begin{theorem}[Transfer from coefficient approximation to forward error]
  \label{thm:combined_transfer}
  Suppose that
  \begin{equation}
    \Norm{D_\ell}_{\leb{2\varrho}(\Xi)}
    \lesssim
    \varepsilon_\ell.
    \label{eq:abstract_coeff_rate}
  \end{equation}
  Then
  \begin{equation}
    \Norm{
      \dgnorm{\ell}{\widetilde A_\ell}{\widetilde c_\ell}{u-u_\ell}
    }_{\leb{\varrho}(\Xi)}
    \lesssim
    \qp{h_\ell^s+\varepsilon_\ell}
    \Norm{\mathscr M_{s,r}}_{\leb{2\varrho}(\Xi)}.
    \label{eq:combined_transfer_bound}
  \end{equation}
\end{theorem}

\begin{proof}
  The pathwise estimate \eqref{eq:strang_lq_compact} gives
  \begin{equation*}
    \dgnorm{\ell}{\widetilde A_\ell}{\widetilde c_\ell}{u-u_\ell}
    \lesssim
    \qp{h_\ell^s+D_\ell}\mathscr M_{s,r}.
  \end{equation*}
  Taking the $\leb{\varrho}(\Xi)$ norm and applying H\"older's
  inequality to $D_\ell\mathscr M_{s,r}$ yields
  \begin{equation*}
    \Norm{
      \dgnorm{\ell}{\widetilde A_\ell}{\widetilde c_\ell}{u-u_\ell}
    }_{\leb{\varrho}(\Xi)}
    \lesssim
    h_\ell^s\Norm{\mathscr M_{s,r}}_{\leb{\varrho}(\Xi)}
    +
    \Norm{D_\ell}_{\leb{2\varrho}(\Xi)}
    \Norm{\mathscr M_{s,r}}_{\leb{2\varrho}(\Xi)}.
  \end{equation*}
  The result follows from \eqref{eq:abstract_coeff_rate}.
\end{proof}

\paragraph{FFT coefficients.}
For the FFT construction define
\begin{equation}
  \varepsilon_{\mathrm{fft}}
  :=
  \Norm{
    \Norm{Y-\widetilde Y_{\mathrm{ref}}}_{\leb q(\W)}
  }_{\leb{2\varrho}(\Xi)}.
  \label{eq:fft_probabilistic_reference_error}
\end{equation}
Since the same reference representation is used on every level,
\begin{equation}
  \Norm{D_\ell}_{\leb{2\varrho}(\Xi)}
  \lesssim
  \varepsilon_{\mathrm{fft}}
  \qquad
  \forall\ell=0,\ldots,L.
  \label{eq:fft_coeff_rate_uniform_levels}
\end{equation}

\begin{corollary}[FFT-based coefficient sampling]
  \label{cor:combined_fft}
  Under the hypotheses of Theorem~\ref{thm:combined_transfer},
  \begin{equation}
    \Norm{
      \dgnorm{\ell}{\widetilde A_\ell}{\widetilde c_\ell}{u-u_\ell}
    }_{\leb{\varrho}(\Xi)}
    \lesssim
    \qp{h_\ell^s+\varepsilon_{\mathrm{fft}}}
    \Norm{\mathscr M_{s,r}}_{\leb{2\varrho}(\Xi)}.
    \label{eq:combined_fft_error_clean}
  \end{equation}
\end{corollary}

For a sequence of increasingly fine target hierarchies, one may choose
$N_{\mathrm{ref}}=N_{\mathrm{ref}}(L)$ so that, for example,
\begin{equation}
  \varepsilon_{\mathrm{fft}}
  \lesssim
  h_L^{\alpha_{\mathrm{fft}}}.
  \label{eq:fft_latent_rate_h}
\end{equation}
The same represented reference field is then used for every level in
that hierarchy.  In particular, $\alpha_{\mathrm{fft}}\ge s$ prevents
reference-field approximation from degrading the forward spatial rate
on the finest level.

Let $u^{\mathrm{ref}}$ denote the continuum forward solution driven by
$\widetilde A_{\mathrm{ref}}
:=\Phi_A\qp{\cdot,\widetilde Y_{\mathrm{ref}}}$ and
$\widetilde c_{\mathrm{ref}}
:=\Phi_c\qp{\cdot,\widetilde Y_{\mathrm{ref}}}$.
Every $u_\ell$ in a fixed FFT hierarchy approximates this same
represented continuum solution.  Thus
$\varepsilon_{\mathrm{fft}}$ contributes to the bias relative to the
target continuum model, while it does not form a level-dependent
representation defect in the correction $Q_\ell-Q_{\ell-1}$.

\paragraph{SPDE coefficients.}
\begin{corollary}[SPDE-based coefficient sampling]
  \label{cor:combined_spde}
  Let $q\in[2,\infty)$ and assume
  \eqref{eq:spde_quasi_uniformity},
  \eqref{eq:dirichlet_h2_regular},
  Lemma~\ref{lem:spde_coeff_from_latent} and the hypotheses of
  Theorem~\ref{thm:combined_transfer}.  Set
  \begin{equation}
    \alpha_{\mathrm{spde}}:=\frac{4-d}{q}.
    \label{eq:combined_spde_rate}
  \end{equation}
  Then
  \begin{equation}
    \Norm{
      \dgnorm{\ell}{\widetilde A_\ell}{\widetilde c_\ell}{u-u_\ell}
    }_{\leb{\varrho}(\Xi)}
    \lesssim
    \qp{h_\ell^s+h_\ell^{\alpha_{\mathrm{spde}}}}
    \Norm{\mathscr M_{s,r}}_{\leb{2\varrho}(\Xi)}.
    \label{eq:combined_spde_error_clean}
  \end{equation}
\end{corollary}

\paragraph{Quantities of interest.}
Since $u_\ell$ is discontinuous, a quantity of interest originally
defined only on $V=\hoz(\W)$ must admit an extension to the discrete
spaces.  We therefore assume that $Q$ is defined on $V+V_\ell$ and
satisfies
\begin{equation}
  \norm{Q(w)-Q(v)}
  \le
  L_Q
  \dgnorm{\ell}{\widetilde A_\ell}{\widetilde c_\ell}{w-v}
  \label{eq:qoi_dg_lipschitz}
\end{equation}
whenever $w\in V$ and $v\in V_\ell$.  Regional averages and bounded
$\leb2$ functionals satisfy such an estimate directly.

It follows that
\begin{equation}
  \Norm{Q(u)-Q(u_\ell)}_{\leb{\varrho}(\Xi)}
  \lesssim
  \Norm{
    \dgnorm{\ell}{\widetilde A_\ell}{\widetilde c_\ell}{u-u_\ell}
  }_{\leb{\varrho}(\Xi)}.
  \label{eq:combined_qoi_error}
\end{equation}
Hence
\begin{equation}
  \Norm{Q(u)-Q(u_\ell)}_{\leb{\varrho}(\Xi)}
  \lesssim
  \qp{h_\ell^s+\varepsilon_{\mathrm{fft}}}
  \Norm{\mathscr M_{s,r}}_{\leb{2\varrho}(\Xi)}
  \label{eq:qoi_fft_error}
\end{equation}
for the FFT construction, while
\begin{equation}
  \Norm{Q(u)-Q(u_\ell)}_{\leb{\varrho}(\Xi)}
  \lesssim
  \qp{h_\ell^s+h_\ell^{(4-d)/q}}
  \Norm{\mathscr M_{s,r}}_{\leb{2\varrho}(\Xi)}
  \label{eq:qoi_spde_error}
\end{equation}
for the SPDE construction.  These strong single-level estimates
provide the bias and correction bounds used in the multilevel
analysis below.

\section{Multilevel estimator and total error decomposition}
\label{sec:mlmc}

We now combine the single-level approximation estimates of
Section~\ref{subsec:combined_error} with the standard MLMC
bias--variance decomposition \cite{Giles2015Acta}.  The two sampling
strategies induce different coefficient approximations: the FFT
construction uses one common represented field throughout a target
hierarchy, while the SPDE construction uses level-dependent latent
fields coupled through consistent restrictions of the same white
noise.

\subsection{Multilevel estimator}
\label{subsec:mlmc_estimator}

Let $\qc{\mathcal T_\ell}_{\ell=0}^L$ be the forward mesh hierarchy,
let $u(\xi)\in V$ denote the continuum solution and let
$u_\ell(\xi)\in V_\ell$ denote the corresponding DG solution.  Define
\begin{equation}
  Q^\star(\xi):=Q\qp{u(\xi)},
  \qquad
  Q_\ell(\xi):=Q\qp{u_\ell(\xi)},
  \qquad
  \ell=0,\ldots,L.
  \label{eq:Qell_def}
\end{equation}
For $\ell\ge1$, the pair $\qp{Q_\ell,Q_{\ell-1}}$ is generated from
one coupled random input: a common reference-field realisation for FFT
and a fine white-noise load together with its restriction
\eqref{eq:noise_coupling} for SPDE.

Set
\begin{equation}
  \delta Q_0:=Q_0,
  \qquad
  \delta Q_\ell:=Q_\ell-Q_{\ell-1},
  \qquad
  \ell=1,\ldots,L.
  \label{eq:DeltaQ_def}
\end{equation}
The prescribed level marginals give
\begin{equation}
  \mathbb E[Q_L]
  =
  \sum_{\ell=0}^L\mathbb E[\delta Q_\ell].
  \label{eq:telescoping_identity}
\end{equation}
The coupling determines the correction variance
$\mathbb V[\delta Q_\ell]$ and hence the efficiency of the multilevel
estimator \cite{ChadaHoelJasraZouraris2022StrongCoupling}.

For independent samples
$\qc{\delta Q_\ell^{(n)}}_{n=1}^{N_\ell}$ on each level, define
\begin{equation}
  \widehat Q_L^{\mathrm{ML}}
  :=
  \sum_{\ell=0}^L
  \frac1{N_\ell}
  \sum_{n=1}^{N_\ell}
  \delta Q_\ell^{(n)}.
  \label{eq:mlmc_estimator}
\end{equation}
Then
\begin{equation}
  \mathbb E\qb{\widehat Q_L^{\mathrm{ML}}}
  =
  \mathbb E[Q_L],
  \qquad
  \mathbb V\qb{\widehat Q_L^{\mathrm{ML}}}
  =
  \sum_{\ell=0}^L
  \frac{\mathbb V[\delta Q_\ell]}{N_\ell}.
  \label{eq:mlmc_var}
\end{equation}

Let $\mathcal C_\ell$ denote the expected work of one sample of
$\delta Q_\ell$, including generation of the coupled random input and
the associated forward solves.  The expected total work is
\begin{equation}
  \mathbb E[\mathrm{Work}]
  =
  \sum_{\ell=0}^L N_\ell\mathcal C_\ell.
  \label{eq:work_total}
\end{equation}
Minimising \eqref{eq:work_total} subject to sampling variance
$\varepsilon_{\mathrm{samp}}^2$ gives
\begin{equation}
  N_\ell
  \propto
  \sqrt{\frac{\mathbb V[\delta Q_\ell]}{\mathcal C_\ell}},
  \label{eq:optimal_allocation}
\end{equation}
with integer allocation
\begin{equation}
  N_\ell
  =
  \left\lceil
    \varepsilon_{\mathrm{samp}}^{-2}
    \qp{
      \sum_{j=0}^L
      \sqrt{\mathbb V[\delta Q_j]\mathcal C_j}
    }
    \sqrt{\frac{\mathbb V[\delta Q_\ell]}{\mathcal C_\ell}}
  \right\rceil.
  \label{eq:optimal_allocation_explicit}
\end{equation}

\subsection{Bias and total error}
\label{subsec:mlmc_bias_variance}
\label{subsec:mlmc_sampling_consequences}

The mean-square error with respect to the continuum quantity
$\mathbb E[Q^\star]$ is
\begin{equation}
  \mathbb E\qb{
    \qp{
      \widehat Q_L^{\mathrm{ML}}
      -
      \mathbb E[Q^\star]
    }^2
  }
  =
  \norm{
    \mathbb E[Q^\star]
    -
    \mathbb E[Q_L]
  }^2
  +
  \sum_{\ell=0}^L
  \frac{\mathbb V[\delta Q_\ell]}{N_\ell}.
  \label{eq:mse_exact_decomposition}
\end{equation}

Let $q,r\in[2,\infty]$ and $s>0$ satisfy
\begin{equation}
  \frac1q+\frac1r+\frac12=1,
  \qquad
  \frac1r<s\le p,
  \label{eq:mlmc_regularities}
\end{equation}
and assume
\begin{equation}
  \mathscr M_{s,r}\in\leb2(\Xi).
  \label{eq:mlmc_regularity_moment}
\end{equation}
The DG-continuity of $Q$ and
Corollaries~\ref{cor:combined_fft} and
\ref{cor:combined_spde} then give the following bias bounds.

\begin{corollary}[FFT bias]
  \label{cor:mlmc_bias_fft}
  For the common-reference FFT construction,
  \begin{equation}
    \norm{
      \mathbb E[Q^\star]
      -
      \mathbb E[Q_L]
    }
    \lesssim
    \qp{
      h_L^s+\varepsilon_{\mathrm{fft}}
    }
    \Norm{\mathscr M_{s,r}}_{\leb2(\Xi)}.
    \label{eq:mlmc_bias_fft}
  \end{equation}
  If the master reference resolution is selected so that
  \begin{equation}
    \varepsilon_{\mathrm{fft}}
    \lesssim
    h_L^{\alpha_{\mathrm{fft}}},
    \label{eq:mlmc_fft_rate}
  \end{equation}
  then
  \begin{equation}
    \norm{
      \mathbb E[Q^\star]
      -
      \mathbb E[Q_L]
    }
    \lesssim
    \qp{
      h_L^s+h_L^{\alpha_{\mathrm{fft}}}
    }
    \Norm{\mathscr M_{s,r}}_{\leb2(\Xi)}.
    \label{eq:mlmc_bias_fft_rate}
  \end{equation}
  In particular, $\alpha_{\mathrm{fft}}\ge s$ gives the spatial bias
  rate $h_L^s$.
\end{corollary}

\begin{remark}[FFT representation error and level corrections]
  \label{rem:fft_representation_cancellation}
  Let $u^{\mathrm{ref}}$ denote the continuum solution driven by
  $\widetilde A_{\mathrm{ref}}
  =\Phi_A\qp{\cdot,\widetilde Y_{\mathrm{ref}}}$ and
  $\widetilde c_{\mathrm{ref}}
  =\Phi_c\qp{\cdot,\widetilde Y_{\mathrm{ref}}}$, and set
  $Q^{\mathrm{ref}}:=Q(u^{\mathrm{ref}})$.  Both $Q_\ell$ and
  $Q_{\ell-1}$ approximate this same represented continuum output.
  Thus $\varepsilon_{\mathrm{fft}}$ controls the bias between
  $Q^{\mathrm{ref}}$ and $Q^\star$, while the coupled correction
  $Q_\ell-Q_{\ell-1}$ measures spatial approximation of
  $Q^{\mathrm{ref}}$.
\end{remark}

\begin{corollary}[SPDE bias]
  \label{cor:mlmc_bias_spde}
  Let $q\in[2,\infty)$ and set
  $\alpha_{\mathrm{spde}}:=(4-d)/q$.  Under the hypotheses of
  Corollary~\ref{cor:combined_spde},
  \begin{equation}
    \norm{
      \mathbb E[Q^\star]
      -
      \mathbb E[Q_L]
    }
    \lesssim
    \qp{
      h_L^s+h_L^{\alpha_{\mathrm{spde}}}
    }
    \Norm{\mathscr M_{s,r}}_{\leb2(\Xi)}.
    \label{eq:mlmc_bias_spde}
  \end{equation}
\end{corollary}

\subsection{Correction variance and computational complexity}
\label{subsec:mlmc_complexity}

Assume that the hierarchy is quasi-geometric,
\begin{equation}
  h_\ell\simeq h_0\rho^{-\ell},
  \qquad
  \rho>1,
  \label{eq:mlmc_geometric_hierarchy}
\end{equation}
and that, for exponents $\alpha,\beta,\gamma>0$,
\begin{equation}
  \norm{\mathbb E[Q^\star-Q_\ell]}
  \lesssim
  h_\ell^\alpha,
  \qquad
  \mathbb V[\delta Q_\ell]
  \lesssim
  h_\ell^\beta,
  \qquad
  \mathcal C_\ell
  \lesssim
  h_\ell^{-\gamma}.
  \label{eq:mlmc_rates}
\end{equation}
The bias estimates above give
\begin{equation}
  \alpha
  =
  \min\qc{s,\alpha_{\mathrm{coeff}}},
  \qquad
  \alpha_{\mathrm{coeff}}
  =
  \begin{cases}
    \alpha_{\mathrm{fft}}, & \text{FFT},\\
    \alpha_{\mathrm{spde}}, & \text{SPDE}.
  \end{cases}
  \label{eq:mlmc_effective_alpha}
\end{equation}
The exponents $\beta$ and $\gamma$ quantify, respectively, the decay
of the coupled correction variance and the growth of coupled-sample
work.

The following standard MLMC complexity result is recalled in the form
used here; see \cite{Giles2015Acta}.

\begin{theorem}[MLMC complexity]
  \label{thm:mlmc_complexity}
  Assume \eqref{eq:mlmc_geometric_hierarchy}--\eqref{eq:mlmc_rates}
  and
  \begin{equation}
    \alpha
    \ge
    \frac12\min\qc{\beta,\gamma}.
    \label{eq:mlmc_rate_compatibility}
  \end{equation}
  Choose $L$ so that
  \begin{equation}
    h_L^\alpha\lesssim\varepsilon
    \label{eq:mlmc_choose_L}
  \end{equation}
  and allocate samples according to
  \eqref{eq:optimal_allocation_explicit} with sampling variance of
  order $\varepsilon^2$.  Then the resulting estimator has
  root-mean-square error of order $\varepsilon$ and expected work
  satisfying
  \begin{equation}
    \mathbb E[\mathrm{Work}]
    \lesssim
    \begin{cases}
      \varepsilon^{-2}, & \beta>\gamma,\\[1mm]
      \varepsilon^{-2}\qp{\log\varepsilon^{-1}}^2,
        & \beta=\gamma,\\[1mm]
      \varepsilon^{-2-(\gamma-\beta)/\alpha},
        & \beta<\gamma.
    \end{cases}
    \label{eq:mlmc_complexity}
  \end{equation}
\end{theorem}

\begin{remark}[Common-reference FFT cost]
  \label{rem:fft_reference_cost}
  For the common-reference FFT construction,
  $N_{\mathrm{ref}}=N_{\mathrm{ref}}(L)$ makes reference-field
  generation a hierarchy-dependent component of the coupled-sample
  cost.  The complete measured costs used in
  Section~\ref{sec:numerics} include reference-field generation
  together with the forward solves.
\end{remark}

\section{Numerical experiments}
\label{sec:numerics}

\newcommand{\numericalscale}{0.75}
\setlength{\textfloatsep}{10pt plus 2pt minus 2pt}
\setlength{\floatsep}{8pt plus 2pt minus 2pt}
\setlength{\intextsep}{10pt plus 2pt minus 2pt}
\captionsetup[figure]{skip=4pt}
\captionsetup[subfigure]{font=footnotesize,labelfont=normalfont,skip=2pt}
\captionsetup{subrefformat=parens}

The experiments examine the spatial accuracy of the polytopic DG
method, approximation of the random coefficients, decay of the
multilevel corrections, and the resulting MLMC work.  All forward
problems are posed on $\W=(0,1)^2$ and use the lifting-based symmetric
interior-penalty method of Section~\ref{sec:polytopic_dg} with
discontinuous $\mathbb P_1$ polynomials and penalty parameter
$\gamma=12$. They are implemented within the package
\cite{evans2026reyna}.

The finest meshes are clipped Voronoi tessellations. Coarser levels
are obtained by exact agglomeration.  The hierarchy used for the field
visualisations contains $4$, $16$, $64$, $256$ and $1024$ cells.  The
finest Voronoi mesh has between four and eight faces per cell, while
the agglomerates have increasingly complex polytopic boundaries.  The
star-ball diagnostic is positive on every element of the displayed
hierarchy.

\begin{table}[htbp]
  \centering
  \scalebox{\numericalscale}{%
    \begin{tabular}{rrrrrr}
\toprule
$\ell$ & cells & $h_\ell$ & DG dofs & $\min_K\rho_K/h_K$ & mean faces\\
\midrule
0 & 4 & 7.1545e-01 & 12 & 7.261e-03 & 94.5\\
1 & 16 & 3.7081e-01 & 48 & 6.983e-03 & 54.1\\
2 & 64 & 1.8818e-01 & 192 & 3.163e-02 & 28.0\\
3 & 256 & 1.0217e-01 & 768 & 1.351e-01 & 13.5\\
4 & 1024 & 5.1823e-02 & 3072 & 2.263e-01 & 5.9\\
\bottomrule
\end{tabular}
  }
  \caption{\em
    Agglomerated Voronoi hierarchy.  The final two columns report the
    minimum chunkiness parameter and the mean number of polytopic
    faces.}
  \label{tab:numerical_hierarchy}
\end{table}

Except in the verification experiment, we solve
\begin{equation}
  -\nabla\cdot\qp{a(x,\xi)\nabla u(x,\xi)}
  +c(x,\xi)u(x,\xi)
  =
  4\pi^2\sin{\pi x_1}\sin{\pi x_2},
  \qquad
  u|_{\partial\W}=0,
  \label{eq:numerical_forward_problem}
\end{equation}
with
\begin{equation}
  a(x,\xi)
  =
  \exp\qp{0.25\Pi_{[-2.5,2.5]}Y(x,\xi)},
  \qquad
  c(x,\xi)
  =
  2\pi^2\exp\qp{0.15\Pi_{[-2.5,2.5]}Y(x,\xi)}.
  \label{eq:numerical_random_coefficients}
\end{equation}
The quantity of interest is the domain average
\begin{equation}
  Q(u):=\frac1{|\W|}\int_\W u(x)\d x.
  \label{eq:numerical_qoi}
\end{equation}

Both random-field constructions use $\kappa=8$ and are normalised to
unit bulk latent variance.  The FFT sampler uses
\begin{equation}
  S(\vec k)\propto\qp{\kappa^2+\norm{\vec k}^2}^{-2}
  \label{eq:numerical_fft_spectrum}
\end{equation}
on a periodic box of side length $1.5$, with the zero Fourier mode
retained.  The SPDE sampler uses
\begin{equation}
  \qp{\kappa^2-\Delta}Y
  =
  \sqrt{4\pi}\,\kappa\,\mathcal W
  \label{eq:numerical_spde_model}
\end{equation}
with homogeneous Dirichlet conditions.  The factor $\sqrt{4\pi}\,\kappa$
gives unit variance for the corresponding whole-space field in two
dimensions.  Thus the two samplers have matched Mat\'ern-type bulk
spectra and identical coefficient maps, although the boundary
condition modifies the SPDE covariance near $\partial\W$ and the two
constructions are not coupled pathwise to one another.

\subsection{Coupled coefficient fields and solutions}
\label{subsec:numerical_hierarchy}

Figure~\ref{fig:numerical_coupling} compares representative
coefficients and forward solutions on the three finest displayed
levels, containing $64$, $256$ and $1024$ polytopes.  Panels
\subref{fig:numerical_coupling_fft} and
\subref{fig:numerical_coupling_spde} show FFT and SPDE sampling,
respectively.

For FFT, one reference-field realisation is evaluated on every
polytopic mesh.  The coefficient is therefore pathwise common across
the hierarchy.  The corresponding solutions converge visibly as the
polytopic mesh is refined.

For SPDE sampling, one fine white-noise load is restricted according
to
$b_{\ell-1}=\qp{\mathbf P_{\ell-1}^{\ell}}^\top b_\ell$,
and the latent SPDE is solved separately on each level.  The
coefficient fields are therefore distinct approximations driven by
the same underlying discrete white noise.

\begin{figure}[!htbp]
  \centering
  \begin{subfigure}[t]{0.49\textwidth}
    \centering
    \makebox[0.485\linewidth]{\footnotesize Diffusion coefficient}\hfill
    \makebox[0.485\linewidth]{\footnotesize DG solution}\par\smallskip
    \includegraphics[width=0.485\linewidth]{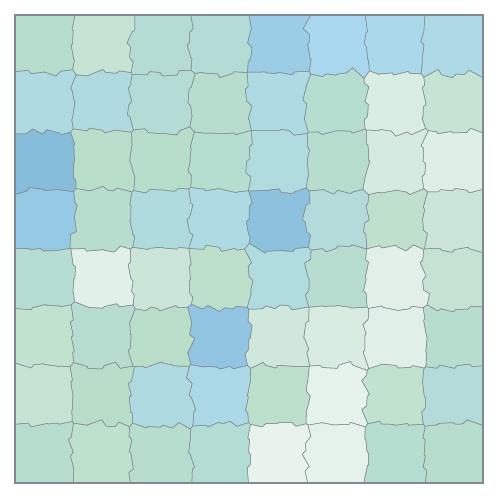}\hfill
    \includegraphics[width=0.485\linewidth]{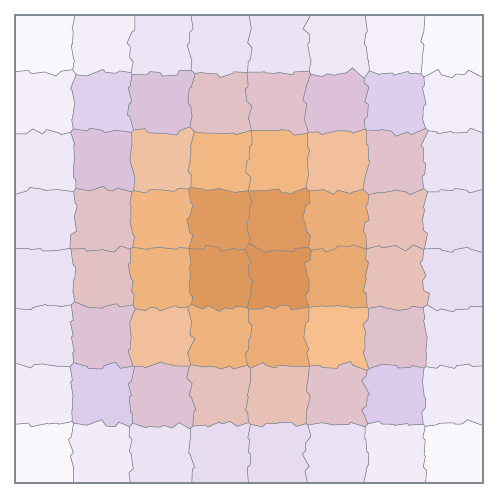}\par\vspace{2pt}
    \includegraphics[width=0.485\linewidth]{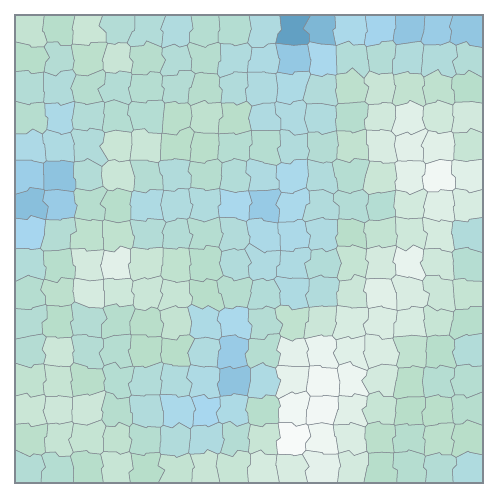}\hfill
    \includegraphics[width=0.485\linewidth]{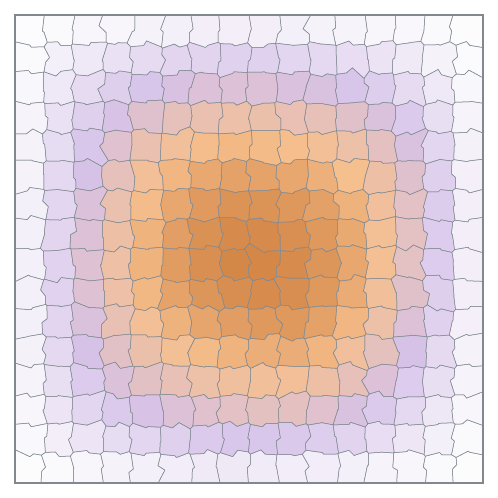}\par\vspace{2pt}
    \includegraphics[width=0.485\linewidth]{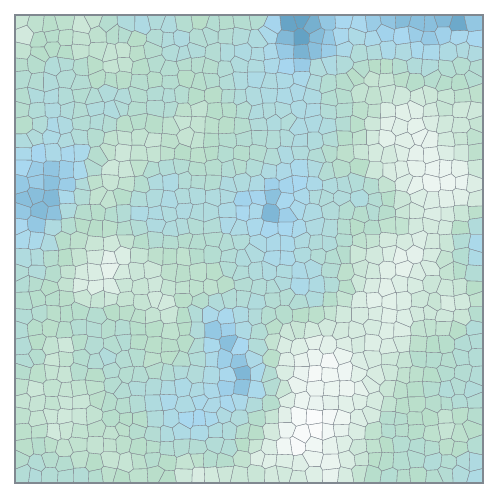}\hfill
    \includegraphics[width=0.485\linewidth]{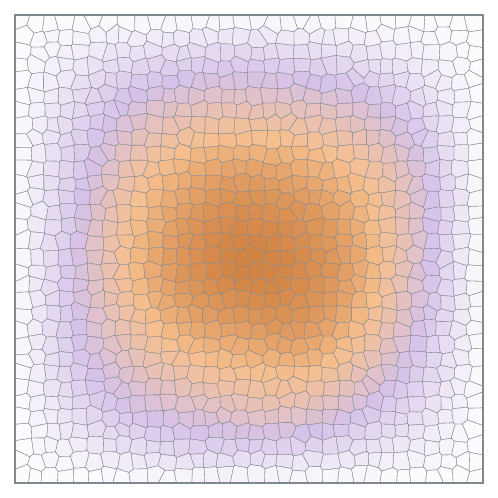}
    \caption{FFT sampler.}
    \label{fig:numerical_coupling_fft}
  \end{subfigure}\hfill
  \begin{subfigure}[t]{0.49\textwidth}
    \centering
    \makebox[0.485\linewidth]{\footnotesize Diffusion coefficient}\hfill
    \makebox[0.485\linewidth]{\footnotesize DG solution}\par\smallskip
    \includegraphics[width=0.485\linewidth]{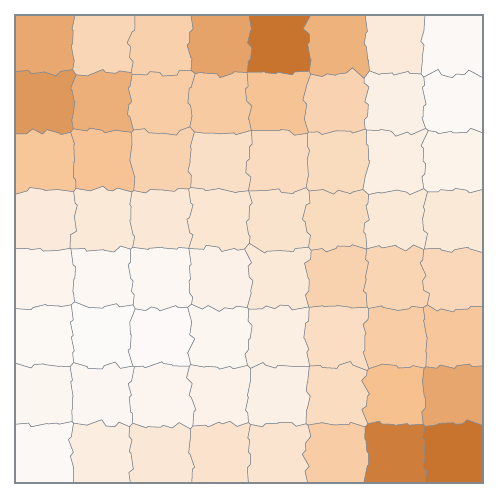}\hfill
    \includegraphics[width=0.485\linewidth]{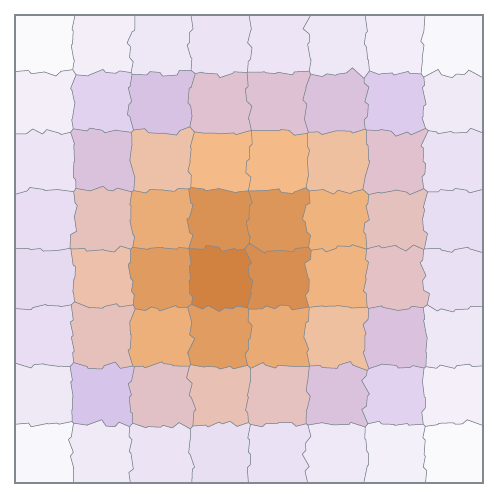}\par\vspace{2pt}
    \includegraphics[width=0.485\linewidth]{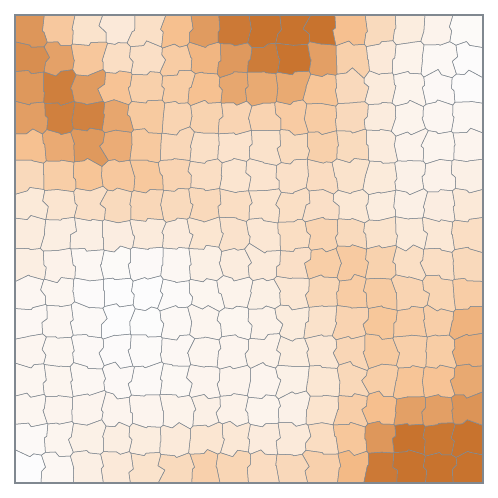}\hfill
    \includegraphics[width=0.485\linewidth]{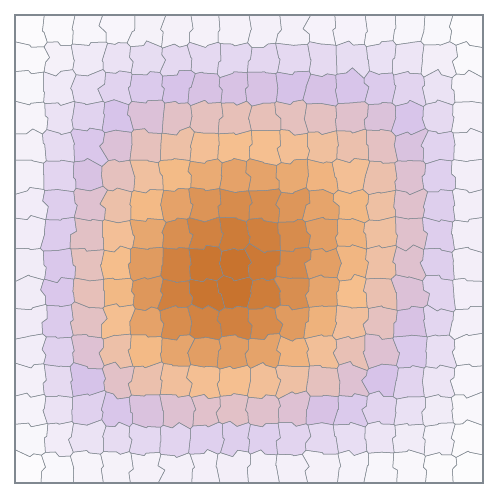}\par\vspace{2pt}
    \includegraphics[width=0.485\linewidth]{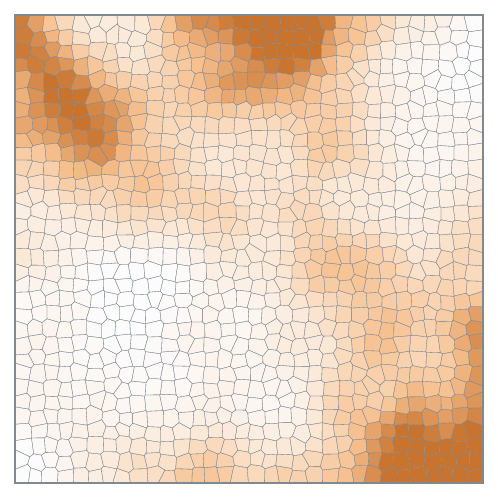}\hfill
    \includegraphics[width=0.485\linewidth]{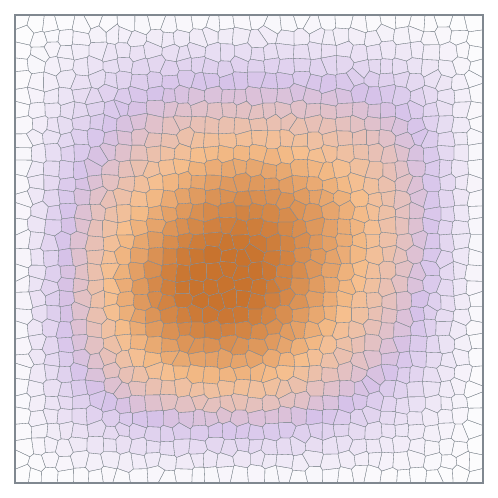}
    \caption{SPDE sampler.}
    \label{fig:numerical_coupling_spde}
  \end{subfigure}
  \caption{\em
    Coupled diffusion coefficients and $P_1$ DG solutions on agglomerated
    Voronoi meshes.  Rows contain $64$, $256$ and $1024$ cells from top
    to bottom; columns within each panel show the coefficient and
    solution.  FFT uses a common reference-field realisation, while
    SPDE uses consistently restricted white-noise loads.}
  \label{fig:numerical_coupling}
\end{figure}

\subsection{Forward discretisation verification}
\label{subsec:numerical_verification}

For spatial verification we solve
\begin{equation}
  -\Delta u+2\pi^2u=4\pi^2u,
  \qquad
  u(x_1,x_2)=\sin{\pi x_1}\sin{\pi x_2},
  \label{eq:numerical_manufactured_solution}
\end{equation}
on independently generated Voronoi meshes with $4$--$1024$ cells;
Figure~\ref{fig:numerical_spatial_verification} reports the resulting
spatial convergence.

\begin{figure}[!htbp]
  \centering
  \includegraphics[width=0.465\textwidth]
    {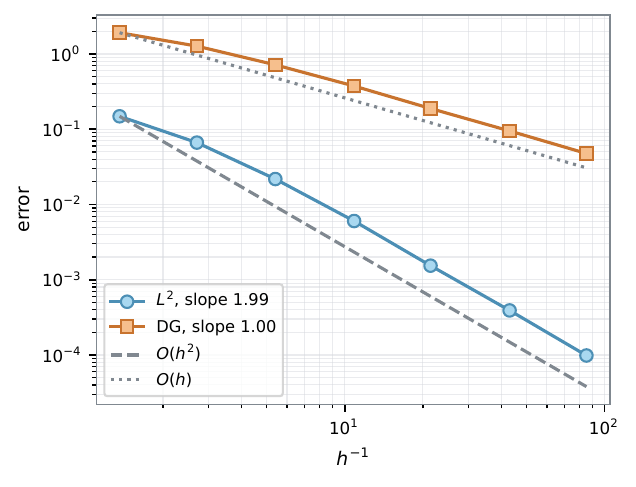}
  \caption{\em
    Spatial convergence on independently generated Voronoi meshes.}
  \label{fig:numerical_spatial_verification}
\end{figure}

\subsection{Coefficient representation}
\label{subsec:numerical_representation}

For the FFT sampler, lower-resolution fields are obtained by spectral
restriction of common $128^2$ master realisations.  The latent and
diffusion errors exhibit approximately first-order decay with the
reference spacing, with fitted exponents close to $1.06$; see
Figure~\ref{fig:numerical_representation}\subref{fig:numerical_fft_representation}.

Refining the forward mesh with a fixed, under-resolved FFT field
produces a representation-error floor.  Refining the random field
together with the forward discretisation removes this floor, directly
illustrating the two contributions in
\eqref{eq:strang_lq_compact}; see
Figure~\ref{fig:numerical_representation}\subref{fig:numerical_fft_strang}.

\begin{figure}[!htbp]
  \centering
  \begin{subfigure}[t]{0.48\textwidth}
    \centering
    \includegraphics[width=0.97\linewidth]
      {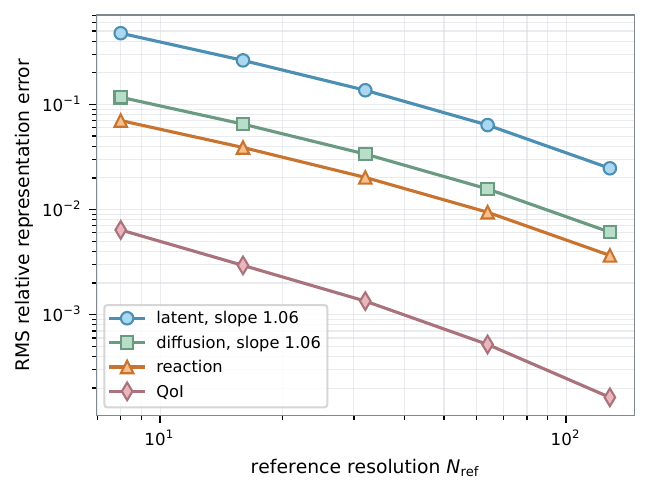}
    \caption{Representation errors.}
    \label{fig:numerical_fft_representation}
  \end{subfigure}\hfill
  \begin{subfigure}[t]{0.48\textwidth}
    \centering
    \includegraphics[width=0.97\linewidth]
      {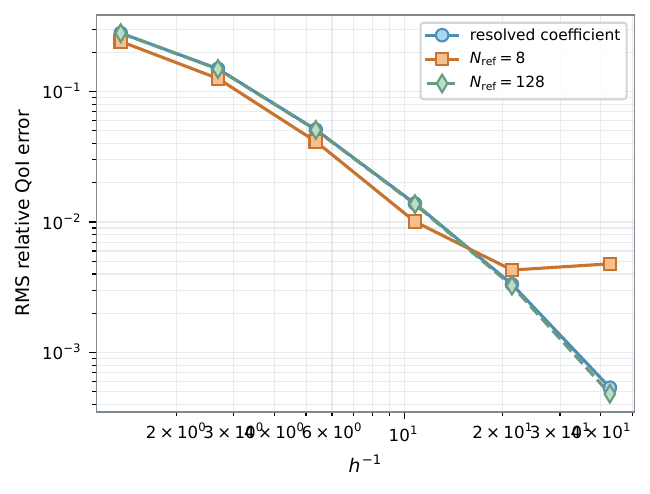}
    \caption{Strang decomposition.}
    \label{fig:numerical_fft_strang}
  \end{subfigure}
  \caption{\em
    FFT coefficient representation: (a) latent-field, diffusion
    coefficient and QoI errors under refinement of a common spectral
    representation; (b) interaction between forward discretisation
    and coefficient-representation errors.}
  \label{fig:numerical_representation}
\end{figure}

The corresponding strong approximation of the SPDE latent field is
studied separately in \cite{pryer2026polytopic}, where the sharp
mean-square rate for the polytopic SIP approximation driven by
canonical discrete white noise is established and verified
numerically.  We therefore focus here on the inter-level coupling and
its effect on the forward multilevel corrections.

\subsection{Multilevel correction decay}
\label{subsec:numerical_variance}

The correction experiment uses the six-level agglomerated hierarchy
reported in Tables~\ref{tab:numerical_variance_fft} and
\ref{tab:numerical_variance_spde}.  For FFT, the fitted weak, variance
and cost exponents are
\begin{equation}
  \widehat\alpha_{\mathrm{fft}}=1.45,
  \qquad
  \widehat\beta_{\mathrm{fft}}=3.12,
  \qquad
  \widehat\gamma_{\mathrm{fft}}=1.27,
  \label{eq:numerical_fft_rates}
\end{equation}
while the SPDE coupling gives
\begin{equation}
  \widehat\alpha_{\mathrm{spde}}=1.39,
  \qquad
  \widehat\beta_{\mathrm{spde}}=2.93,
  \qquad
  \widehat\gamma_{\mathrm{spde}}=1.28.
  \label{eq:numerical_spde_rates}
\end{equation}
Thus $\widehat\beta>\widehat\gamma$ for both coupling strategies.
Figure~\ref{fig:numerical_corrections}, panels
\subref{fig:numerical_weak_decay} and
\subref{fig:numerical_variance_decay}, shows the mean corrections
and correction variances, respectively.

The coupling becomes extremely strong on the finest pair.  At the
precision reported in Tables~\ref{tab:numerical_variance_fft} and
\ref{tab:numerical_variance_spde}, the neighbouring-level
correlations approach one, while
\begin{equation}
  \frac{\mathbb V[\delta Q_\ell]}{\mathbb V[Q_\ell]}
  =
  1.30\times10^{-5}
  \quad\text{for FFT},
  \qquad
  1.42\times10^{-5}
  \quad\text{for SPDE}
  \label{eq:numerical_finest_variance_ratios}
\end{equation}
on the finest pair.  The decay of the variance ratios and $1-\rho$
is shown in
Figure~\ref{fig:numerical_corrections}\subref{fig:numerical_coupling_diagnostic}.

\begin{figure}[!htbp]
  \centering
  \begin{subfigure}[t]{0.32\textwidth}
    \centering
    \includegraphics[width=\linewidth]{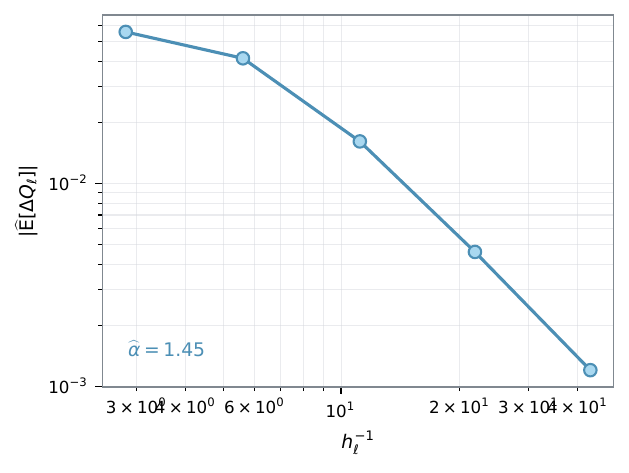}\par\vspace{3pt}
    \includegraphics[width=\linewidth]{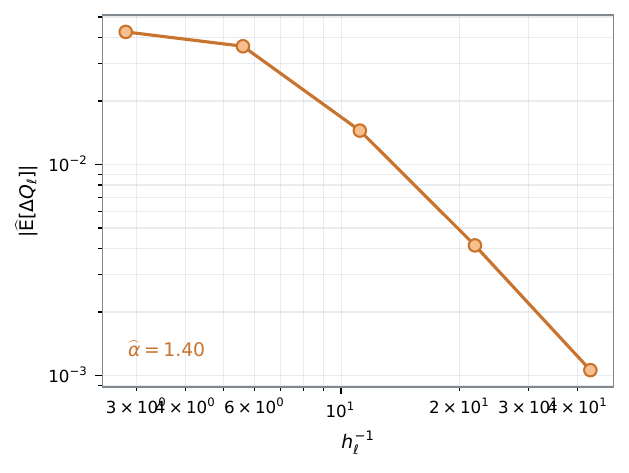}
    \caption{Mean corrections.}
    \label{fig:numerical_weak_decay}
  \end{subfigure}\hfill
  \begin{subfigure}[t]{0.32\textwidth}
    \centering
    \includegraphics[width=\linewidth]{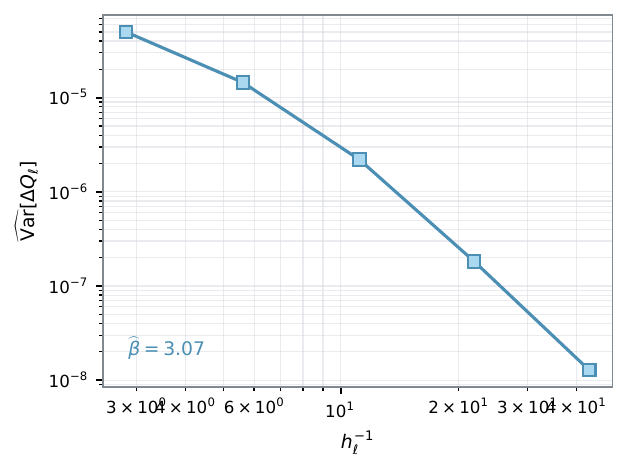}\par\vspace{3pt}
    \includegraphics[width=\linewidth]{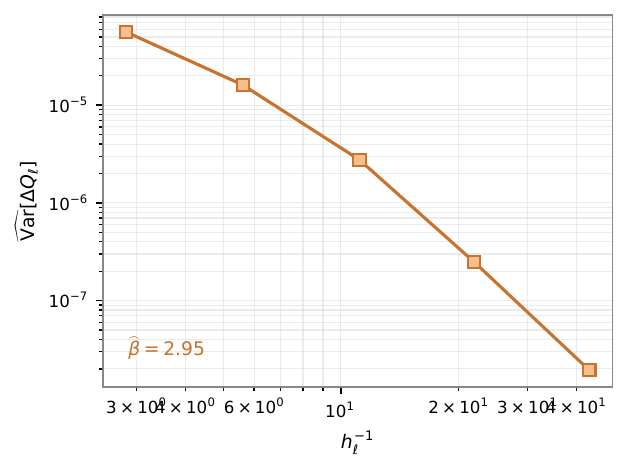}
    \caption{Correction variances.}
    \label{fig:numerical_variance_decay}
  \end{subfigure}\hfill
  \begin{subfigure}[t]{0.32\textwidth}
    \centering
    \includegraphics[width=\linewidth]{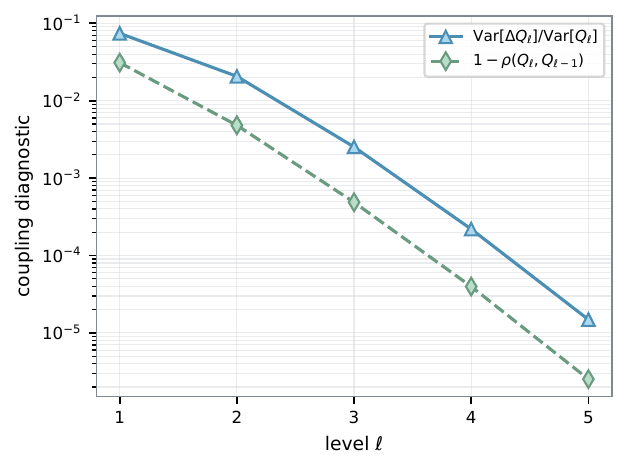}\par\vspace{3pt}
    \includegraphics[width=\linewidth]{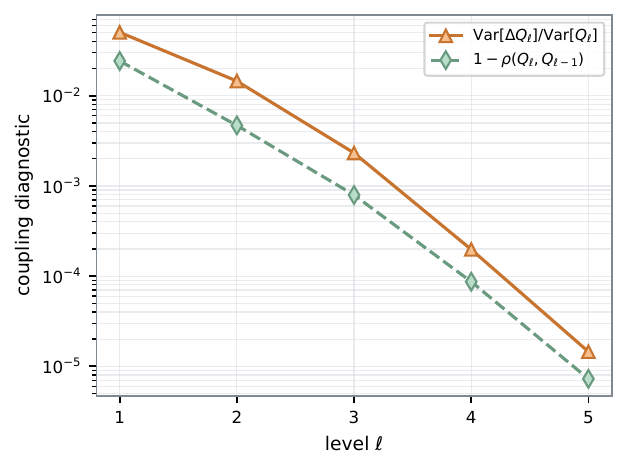}
    \caption{Coupling diagnostics.}
    \label{fig:numerical_coupling_diagnostic}
  \end{subfigure}
  \caption{\em
    Multilevel correction decay and coupling, with FFT in the top row
    and SPDE in the bottom row.  The columns show mean corrections,
    correction variances, and the correction-to-fine variance ratio
    together with $1-\rho$.}
  \label{fig:numerical_corrections}
\end{figure}

\begin{table}[htbp]
  \centering
  \scalebox{\numericalscale}{%
    \begin{tabular}{rrrrrrr}
\toprule
$\ell$ & $h_\ell$ & $|\widehat{\mathrm{E}}[\Delta Q_\ell]|$ & $\widehat{\mathrm{Var}}[\Delta Q_\ell]$ & var. ratio & corr. & cost (s)\\
\midrule
0 & 7.0788e-01 & 2.8478e-01 & 4.8341e-04 & 1.000e+00 & -- & 1.288e-02\\
1 & 3.5449e-01 & 5.5812e-02 & 5.0030e-05 & 7.429e-02 & 0.9689 & 4.553e-02\\
2 & 1.7807e-01 & 4.1370e-02 & 1.4517e-05 & 2.066e-02 & 0.9952 & 1.089e-01\\
3 & 8.9632e-02 & 1.6117e-02 & 2.2005e-06 & 2.528e-03 & 0.9995 & 2.331e-01\\
4 & 4.5588e-02 & 4.6014e-03 & 1.8236e-07 & 2.196e-04 & 1.0000 & 4.933e-01\\
5 & 2.3147e-02 & 1.2034e-03 & 1.2753e-08 & 1.487e-05 & 1.0000 & 1.143e+00\\
\midrule
\multicolumn{7}{l}{Fitted rates: $\alpha=1.45$, $\beta=3.07$, $\gamma=1.27$.}\\
\bottomrule
\end{tabular}
  }
  \caption{\em
    FFT multilevel correction statistics.}
  \label{tab:numerical_variance_fft}
\end{table}

\begin{table}[htbp]
  \centering
  \scalebox{\numericalscale}{%
    \begin{tabular}{rrrrrrr}
\toprule
$\ell$ & $h_\ell$ & $|\widehat{\mathrm{E}}[\Delta Q_\ell]|$ & $\widehat{\mathrm{Var}}[\Delta Q_\ell]$ & var. ratio & corr. & cost (s)\\
\midrule
0 & 7.0788e-01 & 3.0754e-01 & 8.1825e-04 & 1.000e+00 & -- & 1.186e-02\\
1 & 3.5449e-01 & 4.2498e-02 & 5.6280e-05 & 5.057e-02 & 0.9757 & 4.353e-02\\
2 & 1.7807e-01 & 3.6340e-02 & 1.6099e-05 & 1.456e-02 & 0.9953 & 1.006e-01\\
3 & 8.9632e-02 & 1.4478e-02 & 2.7624e-06 & 2.321e-03 & 0.9992 & 2.191e-01\\
4 & 4.5588e-02 & 4.1354e-03 & 2.4814e-07 & 1.975e-04 & 0.9999 & 4.764e-01\\
5 & 2.3147e-02 & 1.0598e-03 & 1.9406e-08 & 1.451e-05 & 1.0000 & 1.171e+00\\
\midrule
\multicolumn{7}{l}{Fitted rates: $\alpha=1.40$, $\beta=2.95$, $\gamma=1.29$.}\\
\bottomrule
\end{tabular}
  }
  \caption{\em
    SPDE multilevel correction statistics.}
  \label{tab:numerical_variance_spde}
\end{table}

\subsection{End-to-end MLMC work}
\label{subsec:numerical_work}

Finally, we compare complete MLMC estimators with finest-level Monte
Carlo using measured sample costs.  For each tolerance, pilot
statistics determine the multilevel allocation and the complete
estimators are then repeated independently to measure their realised
RMSE.

The work--accuracy curves in
Figure~\ref{fig:numerical_mlmc_efficiency}\subref{fig:numerical_mlmc_work}
show that MLMC reduces the measured work for both coefficient samplers
over the full tolerance range considered, with the gain increasing as the
tolerance is tightened.  This behaviour is consistent with the rapid
decay of the coupled correction variances observed above.
Panels~\subref{fig:numerical_mlmc_allocation_fft} and
\subref{fig:numerical_mlmc_allocation_spde} give the FFT and SPDE sample
allocations at the tightest tolerance.

\begin{figure}[!htbp]
  \centering
  \begin{subfigure}[b]{0.36\textwidth}
    \centering
    \includegraphics[width=\linewidth]{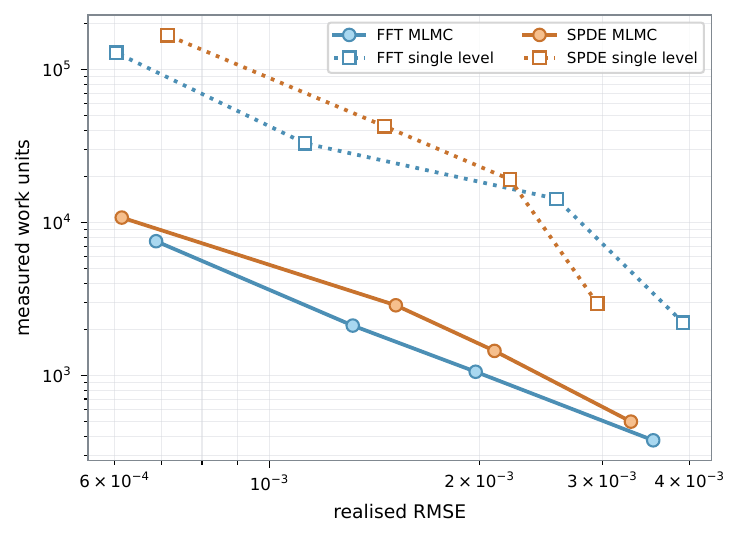}
    \caption{Work against RMSE.}
    \label{fig:numerical_mlmc_work}
  \end{subfigure}\hfill
  \begin{subfigure}[b]{0.30\textwidth}
    \centering
    \includegraphics[width=\linewidth]{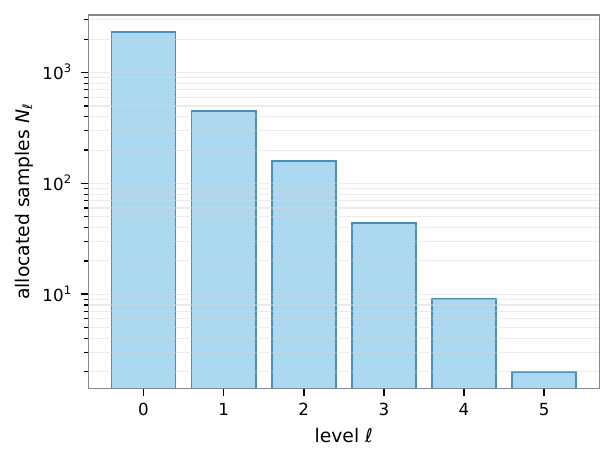}
    \caption{FFT allocation.}
    \label{fig:numerical_mlmc_allocation_fft}
  \end{subfigure}\hfill
  \begin{subfigure}[b]{0.30\textwidth}
    \centering
    \includegraphics[width=\linewidth]{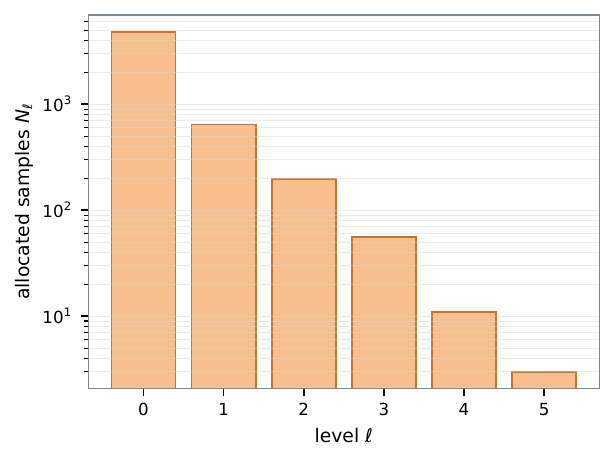}
    \caption{SPDE allocation.}
    \label{fig:numerical_mlmc_allocation_spde}
  \end{subfigure}
  \caption{\em
    Measured MLMC performance.  In (a), solid curves denote MLMC and
    dotted curves finest-level Monte Carlo, with a common colour for
    each sampler.  Panels (b) and (c) show sample allocations at the
    tightest displayed tolerance.}
  \label{fig:numerical_mlmc_efficiency}
\end{figure}

\begin{table}[htbp]
  \centering
  \scalebox{\numericalscale}{%
    \begin{tabular}{lrrrrrr}
\toprule
sampler & target & MLMC RMSE & MLMC work & SLMC RMSE & SLMC work & speed-up\\
\midrule
FFT & 5.0e-03 & 3.545e-03 & 3.762e+02 & 3.914e-03 & 2.210e+03 & 5.87\\
FFT & 3.0e-03 & 1.973e-03 & 1.057e+03 & 2.578e-03 & 1.428e+04 & 13.51\\
FFT & 2.0e-03 & 1.315e-03 & 2.118e+03 & 1.123e-03 & 3.298e+04 & 15.57\\
FFT & 1.0e-03 & 6.875e-04 & 7.545e+03 & 6.029e-04 & 1.289e+05 & 17.08\\
SPDE & 5.0e-03 & 3.294e-03 & 4.986e+02 & 2.948e-03 & 2.955e+03 & 5.93\\
SPDE & 3.0e-03 & 2.100e-03 & 1.447e+03 & 2.211e-03 & 1.911e+04 & 13.21\\
SPDE & 2.0e-03 & 1.516e-03 & 2.875e+03 & 1.461e-03 & 4.262e+04 & 14.83\\
SPDE & 1.0e-03 & 6.139e-04 & 1.077e+04 & 7.136e-04 & 1.677e+05 & 15.57\\
\bottomrule
\end{tabular}
  }
  \caption{\em
    Realised RMSE and measured work for MLMC and finest-level Monte
    Carlo.}
  \label{tab:numerical_mlmc_work}
\end{table}

The experiments confirm the expected $P_1$ polytopic DG rates and
demonstrate the distinct roles of coefficient representation and
inter-level coupling.  The FFT experiments exhibit the
representation-error contribution predicted by the Strang estimate,
while both coupling strategies produce rapidly decaying multilevel
corrections.  The observed correction-variance exponents exceed the
measured growth of coupled-sample cost, and the resulting MLMC
estimators substantially reduce the measured work relative to
finest-level Monte Carlo.

\section{Conclusions}
\label{sec:conclusions}

We have developed a multilevel Monte Carlo framework for coercive
elliptic problems with random diffusion and reaction coefficients on
general polytopic meshes.  The forward problem is discretised by a
lifting-based symmetric interior penalty DG method, and the analysis
separates spatial discretisation error from error introduced by the
representation of the random coefficients.  Two coefficient-generation
strategies were considered: a common-reference FFT construction and an
SPDE-based Mat\'ern construction coupled through projected white-noise
loads on agglomerated hierarchies.  The numerical experiments on
Voronoi and agglomerated polytopic meshes confirm the expected forward
rates, illustrate the coefficient-representation contribution to the
error, and show strong decay of the coupled multilevel corrections.

\printbibliography

\end{document}